\documentclass[11pt]{article}
\usepackage[top=2.0cm,bottom=2cm,left=3cm,right=3cm]{geometry}%页面边距
\usepackage{float}%提供浮动环境
\usepackage{amsmath,amssymb,mathrsfs,amsthm,enumerate,color}
\usepackage[colorlinks,citecolor=red,linkcolor=blue]{hyperref}%交叉引用超链接
\usepackage{booktabs}%三线表
\usepackage{multirow}%提供跨行命令
\newtheorem{theorem}{Theorem}[section]
\newtheorem{lemma}{Lemma}[section]

\newtheorem{example}{Example}[section]
\numberwithin{equation}{section}%公式编号时与章节关联
\usepackage{verbatim}%多行注释
\begin{document}
	\title{\bf Analysis of a WSGD scheme for backward fractional Feynman-Kac equation with nonsmooth data}
\author{Liyao Hao\thanks{Center for Applied Mathematics, Tianjin University, Tianjin 300072, China.}
        \and
        Wenyi Tian\thanks{Corresponding author. Center for Applied Mathematics, Tianjin University, Tianjin 300072, China. This author was partially supported by the National Natural Science Foundation of China (Nos. 12071343 and 11701416). Email: twymath@gmail.com}
    }
    \date{}
	\maketitle
	\begin{abstract} % Gr$\ddot{\text{u}}$nwald
	In this paper, we propose and analyze a second-order time-stepping numerical scheme for the inhomogeneous backward fractional Feynman-Kac equation with nonsmooth initial data. The complex parameters and time-space coupled Riemann-Liouville fractional substantial integral and derivative in the equation bring challenges on numerical analysis and computations. The nonlocal operators are approximated by using the weighted and shifted Gr\"{u}nwald difference (WSGD) formula. Then a second-order WSGD scheme is obtained after making some initial corrections. Moreover, the error estimates of the proposed time-stepping scheme are rigorously established without the regularity requirement on the exact solution. Finally, some numerical experiments are performed to validate the efficiency and accuracy of the proposed numerical scheme.
	
	
	\end{abstract}
%-----------------------------------------------------------------------------------------------------
\section{Introduction}\label{Section1}
%-----------------------------------------------------------------------------------------------------
	
	Let $G(x,\mathbb{A},t)$ denote the joint probability density function of finding the particle on $\mathbb{A}$ at the time $t$ with the initial position of the particle at $x$, and $\mathbf{i}$ the imaginary unit. For particles with power-law waiting time, the governing equation of $G(x,\rho,t)=\int_{-\infty}^{\infty}G(x,\mathbb{A},t)e^{-\mathbf{i}\rho \mathbb{A}}d\mathbb{A}$ is the backward fractional Feynman-Kac equation \cite{CarmiTB:2010,DengHWX:2020} as follows
	\begin{equation}\label{eq:bfFKe}
	\left\{\begin{aligned}
	&\frac{\partial G(x,\rho,t)}{\partial t}={}_0D_{t}^{1-\alpha,x}\Delta G(x,\rho,t)\\
	&\qquad\qquad\qquad
	-\rho U(x)G(x,\rho,t)+f(x,\rho,t),&(x,t)\in \Omega\times(0,T],\\
	&G(x,\rho,0)=G_0(x),&x\in\Omega,\\
	&G(x,\rho,t)=0,&(x,t)\in\partial\Omega\times(0,T],\\
	\end{aligned}\right.
    \end{equation}
	where $\alpha\in (0,1)$, $\Delta$ represents the Laplace operator, $\Omega$ is a bounded convex polygonal domain in $\mathbb{R}^n(n=1,2,3)$ with boundary $\partial\Omega$, the function $U(x)$ is bounded in $\bar{\Omega}$, and $f(x,\rho,t)$ is the source term. ${}_0D_{t}^{\nu,x}$ refers to the Riemann-Liouville fractional substantial derivative \cite{LiDZ:2019}, which is defined by
	\begin{equation}\label{eq:defRLfsd}
	\begin{aligned}
	{}_{0} D_{t}^{\nu, x} G(x, \rho, t) &=\frac{1}{\Gamma(1-\nu)}\left[\frac{\partial}{\partial t}+\rho U(x)\right] \int_{0}^{t}(t-\xi)^{-\nu} e^{-(t-\xi) \rho U(x)} G(x, \rho, \xi) d \xi \\
	&=e^{-t \rho U(x)}{}_{0} D_{t}^{\nu}\left(e^{t \rho U(x)} G(x, \rho, t)\right),\quad\nu\in(0,1),
	\end{aligned}
	\end{equation}
	and ${}_{0} D_{t}^{\nu}$ denotes the Riemann-Liouville fractional derivative \cite{Podlubny:1999} defined by
	\begin{equation}\label{eq:defRLfd}
	{ }_{0} D_{t}^{\nu} G(x, \rho, t)=\frac{1}{\Gamma(1-\nu)} \frac{\partial}{\partial t} \int_{0}^{t}(t-\xi)^{-\nu} G(x, \rho, \xi) d \xi, \quad \nu \in(0,1),
	\end{equation}
	where $\Gamma(s)=\int_{0}^{\infty}t^{s-1}e^{-t}dt$ denotes the Euler gamma function.
	
	The fractional Feynman-Kac equations describe the distribution of the particles' path functionals defined as $\mathbb{A}=\int_{0}^{t}U[x(\tau)]d\tau$ in the anomalous diffusion process, where $x(t)$ is the trajectory of the particle and $U(x)$ has different choices depending on practical applications \cite{DengHWX:2020}. There are two kinds of Feynman-Kac equations: the forward equation and the backward one, and the latter just focuses on the distribution of the functionals. The governing equations for the distribution of the functionals refer to the fractional Feynman-Kac equations if the particles have jump length distributions and/or power-law waiting time, which were derived by Carmi, Turgeman and Barkai \cite{CarmiTB:2010,TurgemanCB:2009}.
	
	A number of efficient numerical algorithms have been designed for solving time-fractional partial differential equations, such as finite difference method \cite{FanLS:2023,JinLZ:2019,LinX:2007,LubichST:1996,StynesOG:2017,YanKF:2018,ZhangSL:2014,ZhouT:2022}, discontinuous Galerkin method \cite{LiLX:2019,MustaphaAF:2014}, spectral method \cite{ChenSZZ:2020,LiZL:2012,LiX:2009,ZhengLTA:2015}, and the references therein. However, the numerical investigations of the fractional Feynman-Kac equations are relatively limited. The challenges on theoretical and numerical issues come from the time-space coupled non-local derivative (fractional substantial derivative) and the complicated parameters in the equations. Chen and Deng \cite{ChenD:2015a} established high-order finite difference approximations for fractional substantial derivatives based on the Lubich method \cite{Lubich:1986b}, which were further applied to numerically solving the forward and backward fractional Feynman-Kac equation \cite{DengCB:2015}.
	For the time fractional substantial diffusion equation, the authors in \cite{HaoCL:2017} designed a compact finite difference scheme with second-order in time and fourth-order in space. The stability and convergence of the scheme are proved for smooth solutions.
	High order difference schemes were proposed in \cite{ChenD:2015b} for the fractional substantial diffusion equation with truncated L\'evy flights and smooth solutions.
	In \cite{DengLQW:2018}, a first-order time-stepping method was provided to solve a forward and tempered fractional Feynman-Kac equation with error analysis in the measure norm.
	Recently, the first-order and second-order time-stepping schemes were introduced in \cite{SunND:2020} for the homogeneous backward fractional Feynman-Kac equation with nonsmooth initial data by using the convolution quadrature approximations of the fractional substantial derivative in \cite{ChenD:2015a}. Moreover, \cite{SunND:2021a} developed high-order fully discrete schemes with some correction terms for the backward fractional Feynman-Kac equation by combing backward difference formulas (BDF) convolution quadrature in time and finite element method in space.
	
	In this work, we dedicate to design new time-stepping scheme for the backward fractional Feynman-Kac equation \eqref{eq:bfFKe} based on the weighted and shifted Gr\"{u}nwald formula proposed in \cite{TianZD:2015}, and discuss the corresponding corrected scheme for the inhomogeneous source term and nonsmooth initial data.
	Our main contributions in this paper are as follows. (i) Based on the weighted and shifted Gr\"{u}nwald formula proposed in \cite{TianZD:2015}, a new corrected second-order time discretization method is proposed for the equation \eqref{eq:bfFKe} with inhomogeneous source term and nonsmooth initial data. (ii) The error estimates of the proposed discrete scheme are also rigorously analyzed.
	
	The structure of the rest of this paper is as follows. Some preliminaries and essential lemmas are introduced and proved in Section \ref{Section2}. In Section \ref{Section3}, we derive a new second order time-stepping scheme for the equation \eqref{eq:bfFKe} by using the weighted and shifted Gr\"{u}nwald formula to approximate the Riemann-Liouville fractional substantial derivative in time with some correction terms. In Section \ref{Section4}, we establish the error estimates of the proposed discrete scheme for the homogeneous and inhomogeneous problem with the nonsmooth initial data. The numerical experiments are provided in order to evaluate the effectiveness and the convergence rates of our proposed numerical scheme in Section \ref{Section5}. We wrap up this paper with some arguments and outlooks in the final section.

%-----------------------------------------------------------------------------------------------------
\section{Preliminaries}\label{Section2}
%-----------------------------------------------------------------------------------------------------
	By using the relationship between the Caputo and Riemann-Liouville fractional derivatives \cite{Podlubny:1999}, the backward fractional Feynman-Kac equation \eqref{eq:bfFKe} can be reformulated as an equivalent form \cite{DengCB:2015,SunND:2021a}, that is
	\begin{equation}\label{2.1}
	\left\{\begin{aligned}
	&{}_0^CD_{t}^{\alpha,x}G(x,\rho,t)-\Delta G(x,\rho,t) ={}_0I_{t}^{1-\alpha,x}f(x,\rho,t),&(x,t)\in \Omega\times(0,T],\\
	&G(x,\rho,0)=G_0(x)&x\in\Omega,\\
	&G(x,\rho,t)=0&(x,t)\in\partial\Omega\times(0,T],\\
	\end{aligned}\right.
	\end{equation}
	where ${}_0I_{t}^{1-\alpha,x}$ denotes Riemann-Liouville fractional substantial integral \cite{LiDZ:2019}, which is defined by
	$$
	{}_0I_{t}^{1-\alpha,x}f(t)=\frac{1}{\Gamma(1-\alpha)}\int_{0}^{t}(t-s)^{-\alpha}e^{-\rho U(x)(t-s)}f(s)ds,~\alpha\in(0,1).
	$$
    The notation ${}_0^CD_{t}^{\alpha,x}$ represents the Caputo fractional substantial derivative  \cite{LiDZ:2019} defined by
    $$
    {}_0^CD_{t}^{\alpha,x}G(x,\rho,t)=e^{-t\rho U(x)}{}_0^CD_{t}^{\alpha}\big(e^{t\rho U(x)}G(x,\rho,t)\big),~\alpha\in(0,1),
    $$
    and ${}_{0}^C D_{t}^{\alpha}$ stands for the Caputo fractional derivative \cite{Podlubny:1999}
    \begin{equation*}
    { }_{0}^C D_{t}^{\alpha} G\left(x, \rho, t\right)=\frac{1}{\Gamma(1-\alpha)}  \int_{0}^{t}(t-\xi)^{-\alpha}\frac{\partial}{\partial \xi} G\left(x, \rho, \xi\right) d \xi, \ \alpha \in(0,1).
    \end{equation*}
    It is mentioned in \cite{SunND:2021a} that the separation of the operators ${}_0D_{t}^{1-\alpha,x}$ and $\Delta$ in \eqref{2.1} compared with \eqref{eq:bfFKe} can reduce the influences of the regularity of $U(x)$ on convergence rate in space. Hence we consider to establish the numerical scheme based on \eqref{2.1} instead of \eqref{eq:bfFKe} in the following.
	The notations $G(t)$, $G_0$, $f(t)$ refer to the abbreviations of $G(x, \rho, t)$, $G_0(x)$, $f(x,\rho,t)$ respectively, and $\|\cdot\|$ denotes the $L^2$ norm and the operator norm from $L^2(\Omega)$ to $L^2(\Omega)$. Throughout this paper, the generic constant $C>0$ may be different at different occurrences but it is always independent of the time step size $\tau$, and $\epsilon> 0$ is arbitrary small.

	We briefly review the Laplace transform of the fractional substantial integral and derivative.
	\begin{lemma}[\cite{LiDZ:2019}]\label{l2.1}
	The Laplace transform of the Riemann-Liouville fractional substantial integral with $\alpha\in(0,1)$ is
	$$
	  {}_0\widetilde{I_t^{\alpha,x}G}(z)=\big(\beta(z)\big)^{-\alpha}\widetilde{G}(z),
	$$
	the Laplace transform of the Riemann-Liouville fractional substantial derivative with $\alpha\in(0,1)$ is
	$$
	{}_0\widetilde{D_t^{\alpha,x}G}(z)=\big(\beta(z)\big)^\alpha\widetilde{G}(z)-{}_0I_t^{1-\alpha}\big(e^{\rho U(x)t}G(t)\big)|_{t=0},
	$$
	and the Laplace transform of the Caputo fractional substantial derivative with $\alpha\in(0,1)$ is
	$$
	{}_0^C\widetilde{D_t^{\alpha,x}G}(z)=\big(\beta(z)\big)^\alpha\widetilde{G}(z)-\big(\beta(z)\big)^{\alpha-1}G(0),
	$$
	where
	\begin{equation}\label{2.2}
	\beta(z):=\beta(z,x):=z+\rho U(x),
	\end{equation}
	`$\sim$' means taking the Laplace transform and ${}_0I_t^{1-\alpha}$ stands for the Riemann-Liouville fractional integral.
	\end{lemma}
	
	Next we define sectors $\Sigma_\theta$ and $\Sigma_{\theta,\kappa}$ with $\kappa>0$ and $\theta\in(\pi/2,\pi)$ in the complex plane $\mathbb{C}$ as
	\begin{align}
	  &\Sigma_{\theta}=\big\{z \in \mathbb{C} \backslash\{0\}:|\arg z| \leq \theta\big\}, \\
	  &\Sigma_{\theta, \kappa}=\big\{z \in \mathbb{C}:|z| \geq \kappa,|\arg z| \leq \theta\big\}.\label{eq:Sigtk}
	\end{align}
	The contour $\Gamma_{\theta,\kappa}$ oriented with an increasing imaginary part is defined by
	\begin{equation}\label{eq:Gamtk}
	\Gamma_{\theta,\kappa}=\big\{z\in\mathbb{C}:|z|=\kappa,|\arg z| \leq \theta\big\}\cup\big\{z\in\mathbb{C}:|z|\geq\kappa,~|\arg z|=\theta\big\}.
	\end{equation}
	In addition, $\Gamma_{\theta,\kappa}^\tau$ is given by
	\begin{equation}\label{eq:Gamtkt}
	\Gamma_{\theta,\kappa}^\tau=\big\{z\in\mathbb{C}:|z|=\kappa,~|\arg z| \leq \theta\big\}\cup\big\{z\in\mathbb{C}:\kappa\leq|z|\leq\frac{\pi}{\tau \sin\theta},~|\arg z|=\theta\big\}.
	\end{equation}
	
	According to Lemma \ref{l2.1}, the Laplace transform of \eqref{2.1} with respect to time follows
	\begin{equation}\label{2.3}
	\widetilde{G}(z)=\big(\beta (z)^\alpha+A\big)^{-1}\beta (z)^{\alpha-1}G_0+\big(\beta(z)^\alpha+A\big)^{-1}\beta(z)^{\alpha-1}\tilde{f}(z),
	\end{equation}
	where $\beta(z)$ is given by \eqref{2.2}.
	Then taking the inverse Laplace transform on both sides of \eqref{2.3} implies that the solution to \eqref{2.1} can be formulated as
	\begin{equation}\label{4.2}
	\begin{aligned}
	G(t)&=\frac{1}{2\pi \mathbf{i}}\int_{\Gamma_{\theta ,\kappa }}e^{zt}\big(\beta (z)^{\alpha }+A\big)^{-1}\beta (z)^{\alpha-1}\big(G_0+\tilde{f}(z)\big)dz.
	\end{aligned}
	\end{equation}

	Next we introduce the following lemma about $\beta(z)$, which is important in the error estimates in Section \ref{Section4}.
	\begin{lemma}[\cite{DengLQW:2018}]\label{l2.2}
	Let $\beta(z)$ be defined in \eqref{2.2} and $U(x)$ be bounded in $\bar{\Omega}$. By choosing $\theta\in(\frac{\pi}{2},\pi)$ sufficiently close to $\frac{\pi}{2}$ and $\kappa>0$ sufficiently large (depending on the value $|\rho|\|U(x)\|_{L^\infty(\bar{\Omega})}$), then
	for $z\in\Sigma_{\theta,\kappa}$, we have
	$$
	\begin{aligned}
	& C_1|z|\le |\beta (z)|\le C_2|z|, \\
	& \left\| A(\beta (z)^{\alpha }+A)^{-1} \right\|\le C, \\
	& \left\| (\beta(z)^{\alpha}+A)^{-1} \right\|\le C| z |^{-\alpha},
	\end{aligned}
	$$
	where $A=-\Delta$, $\Sigma_{\theta,\kappa}$ is defined by \eqref{eq:Sigtk}, and $C, C_1, C_2$ are positive constants.
	\end{lemma}
	
    The regularity estimate on the solution of \eqref{2.1} has been derived in \cite{SunND:2021a}, the result is stated in the following theorem.
    \begin{theorem}[\cite{SunND:2021a}]\label{t2.1}
    Assume that $G(t)$ is the solution of \eqref{2.1} and $U(x)$ is bounded in $\bar{\Omega}$. If $G_0\in L^2(\Omega)$ and $\int_0^t(t-s)^{-q \alpha / 2}\|f(s)\|_{L^2(\Omega)} d s<\infty$, then we have
    $$
    \|G(t)\|_{\dot{H}^q(\Omega)} \leq C t^{-q \alpha / 2}\left\|G_0\right\|_{L^2(\Omega)}+C \int_0^t(t-s)^{-q \alpha / 2}\|f(s)\|_{L^2(\Omega)}ds, \ q \in[0,2],
    $$
    where $\|G\|_{\dot{H}^q(\Omega)}=\big(\sum_{j=1}^{\infty}\lambda_j^q(G,\varphi_j)^2\big)^{1/2}$ with $\{(\lambda_j, \varphi_j)\}_{j=1}^\infty$ being the eigenvalues and the normalized eigenfunctions of operator $A=-\Delta$ with a zero Dirichlet condition.
    \end{theorem}

    In the following, we derive some lemmas, which play crucial roles in the analysis of the error estimates in Section \ref{Section4}.

	\begin{lemma}\label{l2.3}
	Let $a_\tau(\zeta )=(1-\zeta)(1+\frac{\alpha}{2}-\frac{\alpha }{2}\zeta )^{\frac{1}{\alpha}}/{\tau}$ and $\beta(z)$ be defined in \eqref{atauzeta} and \eqref{2.2}, respectively. Then we have
	$$
	C_1|z|\le |a_{\tau}(e^{-\beta(z) \tau })|\le C_2|z|,\ \forall~ z\in \Gamma_{\theta ,\kappa }^{\tau },
	$$
	where $\Gamma_{\theta,\kappa}^{\tau}$ is given by \eqref{eq:Gamtkt}, and $C_1, C_2$ denote two positive constants.
	\end{lemma}
	\begin{proof}
	Let $\omega:=\beta(z)$, then the first estimate in Lemma \ref{l2.2} shows that
    $$
    C_1|z|\tau\le\left| \omega \tau  \right|\le C_2\left| z \right|\tau \le \frac{C_2\pi }{\sin \theta },~\forall~ z\in \Gamma_{\theta ,\kappa }^{\tau },
    $$
    then it suffices to prove that $C_1|\omega |\le |a_{\tau}(e^{-\omega \tau})|\le C_2|\omega |$ holds for any $z\in \Gamma_{\theta ,\kappa }^{\tau }$.
	Let $\gamma(\zeta):=(1-\zeta)^\alpha(1+\frac{\alpha}{2}-\frac{\alpha}{2}\zeta)$, we first prove that $\frac{|\omega|}{| a_{\tau}(e^{-\omega \tau})|}=\frac{| \omega \tau  |}{| \gamma(e^{-\omega \tau})^{\frac{1}{\alpha}} |}$ is bounded for any $z\in \Gamma_{\theta ,\kappa }^{\tau }$ by the similar approach in \cite{WangYYP:2020}.
	
	(i) For the case $0\le | \omega \tau |\le \delta_0$ with some $0<\delta_0\le \frac{C_2\pi }{\sin \theta }$, we have that
	$$
	\lim\limits_{x\rightarrow 0}\frac{x}{\gamma(e^{-x})^{\frac{1}{\alpha}}}=\lim\limits_{x\rightarrow 0}\frac{x}{\left(x^\alpha+d_1 x^{2+\alpha}+d_2 x^{3+\alpha}+\cdots \right)^{\frac{1}{\alpha}}}\\
	=\lim\limits_{x\rightarrow 0}\frac{1}{\left(1+d_1x^2+\cdots\right)^{\frac{1}{\alpha}}}=1,
	$$
	which implies the boundedness of $\frac{| \omega \tau |}{| \gamma(e^{-\omega \tau})^{\frac{1}{\alpha}} |}$ when $0\le | \omega \tau |\le \delta_0$ and $z\in \Gamma_{\theta ,\kappa }^{\tau }$.
	
	(ii) For the case $\delta_0\le |\omega \tau|\le \frac{C_2\pi }{\sin \theta }$, the continuity of $\frac{| \omega \tau  |}{| \gamma(e^{-\omega \tau})^{\frac{1}{\alpha}} |}$ at any $\omega \tau \ne 0$ and $z\in \Gamma_{\theta ,\kappa }^{\tau }$ leads to the boundedness of $\frac{| \omega \tau  |}{| \gamma(e^{-\omega \tau})^{\frac{1}{\alpha}} |}$ when $\delta_0\le |\omega \tau|\le \frac{C\pi }{\sin \theta }, z\in \Gamma_{\theta ,\kappa }^{\tau }$.
	
	Combing (i) and (ii), we obtain that $\frac{|\omega|}{|a_{\tau}(e^{-\omega \tau}) |}=\frac{|\omega \tau|}{| \gamma{{(e^{-\omega \tau })}^{\frac{1}{\alpha }}} |}$ is bounded for all $z\in \Gamma_{\theta ,\kappa }^{\tau }$, which implies $C_1|\omega |\le |a_{\tau}(e^{-\omega \tau })|$. Similarly, we can derive that $\frac{\left|a_{\tau}(e^{-\omega \tau}) \right|}{\left|\omega\right|}$ is also bounded for all $z\in \Gamma_{\theta ,\kappa }^{\tau }$, which leads to $|a_{\tau}(e^{-\omega \tau })|\le C_2|\omega |$. This completes the proof.
	\end{proof}
	
	\begin{lemma}\label{l2.4}
	Let $a_\tau(\zeta )=(1-\zeta)(1+\frac{\alpha}{2}-\frac{\alpha }{2}\zeta )^{\frac{1}{\alpha}}/{\tau}$ be defined in \eqref{atauzeta}, $\zeta =e^{\mathbf{i}\theta }$ and $\theta \in [0,2\pi]$. Then we have
	$$
	a_{\tau}(e^{\mathbf{i}\theta})^{\alpha}\in \Sigma_{\frac{\alpha \pi }{2}},\ \alpha \in (0,1).
	$$
	\end{lemma}
	\begin{proof}
	From the definition of $a_{\tau}(\cdot)$ and $\theta \in (0,2\pi )$, it obtains the argument of $a_{\tau}(e^{\mathbf{i}\theta})^{\alpha}$ as follows
	$$
	\begin{aligned}
	\arg a_\tau(e^{\mathbf{i}\theta})^{\alpha}& =\arg(1-e^{\mathbf{i}\theta})^{\alpha}+\arg\big(1+\frac{\alpha}{2}-\frac{\alpha}{2}{e}^{\mathbf{i}\theta}\big) \\
	& =\alpha \arg (1-\cos \theta -\mathbf{i}\sin \theta )+\arg\big(1+\frac{\alpha }{2}-\frac{\alpha }{2}\cos \theta -\mathbf{i}\frac{\alpha }{2}\sin \theta\big) \\
	& =\alpha \arctan\Big(\frac{-\sin \theta }{1-\cos \theta }\Big)+\arctan\Big(\frac{-\frac{\alpha }{2}\sin \theta }{1+\frac{\alpha }{2}-\frac{\alpha }{2}\cos \theta }\Big) \\
	& =\alpha\Big(\frac{\theta }{2}-\frac{\pi }{2}\Big)+\arctan\Big(\frac{-\frac{\alpha }{2}\sin \theta }{1+\frac{\alpha }{2}-\frac{\alpha }{2}\cos \theta }\Big):=\varphi (\theta ).
	\end{aligned}
	$$
	The derivative of $\varphi (\theta )$ satisfies
	$$
	\begin{aligned}
	  \varphi '(\theta )&=\frac{\alpha }{2}+\frac{1}{1+{{\Big(\frac{-\frac{\alpha }{2}\sin \theta }{1+\frac{\alpha }{2}-\frac{\alpha }{2}\cos \theta }\Big)}^{2}}}\cdot\frac{-\frac{\alpha }{2}\cos \theta (1+\frac{\alpha }{2}-\frac{\alpha }{2}\cos \theta )-(-\frac{\alpha }{2}\sin \theta )\frac{\alpha }{2}\sin \theta }{{{(1+\frac{\alpha }{2}-\frac{\alpha }{2}\cos \theta )}^{2}}} \\
	  & =\frac{\alpha }{2}+\frac{-\frac{\alpha }{2}(1+\frac{\alpha }{2})\cos \theta+\frac{{{\alpha }^{2}}}{4}}{{{(1+\frac{\alpha }{2}-\frac{\alpha }{2}\cos \theta )}^{2}}+{{(-\frac{\alpha }{2}\sin \theta )}^{2}}} \\
	  & =\frac{\frac{\alpha }{2}(1+\frac{\alpha }{2})(1+\alpha )(1-\cos \theta )}{\big(1+\frac{\alpha}{2}-\frac{\alpha }{2}\cos\theta\big)^2+\big(-\frac{\alpha }{2}\sin \theta\big)^2}\ge0. \\
	\end{aligned}
	$$
    This implies that $\varphi (\theta )$ is monotonically increasing for $\theta\in[0,2\pi]$, then $-\frac{\alpha \pi }{2}=\varphi (0)\le\varphi(\theta)\le \varphi (2\pi)=\frac{\alpha \pi }{2}$. Thus we have $a_{\tau}(e^{\mathbf{i}\theta})^{\alpha}\in \Sigma_{\frac{\alpha \pi }{2}}.$ %\qedhere
	\end{proof}
	
	\begin{lemma}\label{l2.5}
	Let $a_\tau(\zeta )=(1-\zeta)(1+\frac{\alpha}{2}-\frac{\alpha }{2}\zeta )^{\frac{1}{\alpha}}/\tau$ be defined in \eqref{atauzeta}, $\beta(z)$ be defined in \eqref{2.2} and $\alpha \in (0,1)$. We have that
	$$
	a_{\tau}(e^{-\beta(z) \tau})^{\alpha}\in \Sigma_{\phi}, ~\forall~z\in \Gamma_{\theta ,\kappa }^{\tau },
	$$
	for some $\phi \in (\alpha \pi /2,\pi )$, where $\tau\le\tau_*$ with $\tau_*$ being a positive constant. Moreover, the operator $\big(a_{\tau}(e^{-\beta(z) \tau})^{\alpha}+A\big)^{-1}$ is analytic with respect to $z$ in the region $\Sigma_{\theta ,\kappa }^{\tau }$ and satisfies
	$$
	\big\|\big(a_{\tau}(e^{-\beta(z) \tau})^{\alpha}+A \big)^{-1} \big\|\le C|a_{\tau}(e^{-\beta(z) \tau})|^{-\alpha}\le C|z|^{-\alpha},~\forall~z\in \Gamma_{\theta ,\kappa }^{\tau }.
	$$
	\end{lemma}
	\begin{proof}
	Let $\omega:=\beta(z)$. Firstly, for sufficiently small step $\tau<\frac{\pi}{2|\rho|\|U\|_{L^\infty(\bar{\Omega})}}$, we have
	$$
	0\le\tau|\mathrm{Im}(\omega)|\le \tau(|\mathrm{Im}(z)|+|\rho|\|U\|_{L^\infty(\bar{\Omega})})\leq\pi+\tau(|\rho|\|U\|_{L^\infty(\bar{\Omega})})<\frac{3}{2}\pi, ~\forall~z\in \Gamma_{\theta ,\kappa }^{\tau }.
	$$
	Setting $s:=\tau|\operatorname{Im}(\omega )|=| \omega |\tau\sin \varphi$, then $s\in [0,\frac{3\pi }{2})$ and $e^{-\omega \tau}=e^{-\tau \left| \omega  \right|(\cos \varphi +\mathbf{i}\sin \varphi )}=e^{-s\cot \varphi }e^{-\mathbf{i}s}$. This and $\beta_{\tau}(\zeta ):=a_{\tau}(\zeta )^{\alpha }$ imply that
	$$
	\big|a_{\tau}(e^{-\omega \tau})^{\alpha }-a_{\tau}(e^{-\mathbf{i}s})^{\alpha} \big|=\big|\beta_{\tau}(e^{-\omega\tau})-\beta_{\tau}(e^{-\mathbf{i}s}) \big|\le Cs| \cot \varphi |\cdot| \beta_{\tau }^{\prime }(e^{-\sigma (s)s\cot (\varphi )}e^{-\mathbf{i}s}) |,
	$$
	where $\sigma(s)\in (0,1)$.
	Further, it obtains
	$$
	\beta_{\tau}^{\prime}(\zeta )=\frac{-\frac{\alpha }{2}{(1-\zeta )^{\alpha }}-\alpha(1+\frac{\alpha }{2}-\frac{\alpha }{2}\zeta ){(1-\zeta )^{\alpha -1}}}{\tau^{\alpha }}.
	$$
	Let $\varphi$ be sufficiently close to $\pi/2$, thus $e^{-\sigma(s)s\cot\varphi}\approx 1$, which leads to
	$$
	|\beta_{\tau}^{\prime}(e^{-\sigma(s)s\cot\varphi}e^{-\mathbf{i}s})|\le  C\tau^{-\alpha}|1-e^{-\sigma(s)s\cot\varphi}e^{-\mathbf{i}s}|^{\alpha-1}.
	$$
	Combing the preceding two estimates with the inequality $|\cot\varphi|\le C|\varphi-\pi/2|$ in \cite{JinLZ:2018a} follows that
    $$
    |\beta_{\tau}(e^{-\omega\tau})-\beta_{\tau}(e^{-\mathbf{i}s})|\le C\tau^{-\alpha}|\varphi-\pi/2|s|1-e^{-\sigma(s)s\cot\varphi}e^{-\mathbf{i}s}|^{\alpha-1}.
    $$
    Recall that
    $$
    \beta_\tau(e^{-\mathbf{i}s})=\frac{2^\alpha(\sin\frac{s}{2})^\alpha\rho(s)}{\tau^{\alpha}}e^{\mathbf{i}(\frac{\alpha\pi}{2}-\frac{\alpha s}{2}+\Psi(s))},
    $$
    where
    $$
    \rho(s)=\sqrt{\big(1+\frac{\alpha}{2}\big)^2+\frac{\alpha^2}{4}-\alpha\big(1+\frac{\alpha}{2}\big)\cos s},~~\Psi(s)=\arctan\frac{\frac{\alpha}{2}\sin s}{1+\frac{\alpha}{2}-\frac{\alpha}{2}\cos s}.
    $$
    If $s\in(0,3\pi/2)$ is small, then Taylor expansion yields $\beta_\tau(e^{-\mathbf{i}s})\approx \tau^{-\alpha}s^\alpha e^{\mathbf{i}\alpha\pi/2}$ and $1-e^{-\sigma(s)s\cot\varphi}e^{-\mathbf{i}s}\approx \sigma(s)s\cot\varphi+\mathbf{i}s$ asymptotically. Immediately, we have
    $$
    |\beta_{\tau}(e^{-\omega\tau})-\beta_{\tau}(e^{-\mathbf{i}s})|\le C\tau^{-\alpha}|\varphi-\pi/2|s^\alpha\le C|\theta-\pi/2||\beta_\tau(e^{-\mathbf{i}s})|.
    $$
    If $s\in(0,3\pi/2)$ is away from $0$, then $|\beta_\tau(e^{-\mathbf{i}s})|\ge C\tau^{-\alpha}$, hence
    $$
    |\beta_{\tau}(e^{-\omega\tau})-\beta_{\tau}(e^{-\mathbf{i}s})|\le C\tau^{-\alpha}|\varphi-\pi/2|\le C|\varphi-\pi/2||\beta_\tau(e^{-\mathbf{i}s})|.
    $$
	By choosing $\varphi $ sufficiently close to $\frac{\pi }{2}$ and using Lemma \ref{l2.4}, we obtain $a_{\tau}(e^{-\omega \tau})^{\alpha}\in {{\Sigma }_{\phi }}$ for some $\phi\in(\alpha\pi/2,\pi)$. The resolvent estimate follows immediately from Lemmas \ref{l2.2} and \ref{l2.3}.
	\end{proof}
	
	\begin{lemma}\label{l2.6}
	Let
	\begin{equation}\label{muzeta}
	\mu (\zeta ):=\big(1+\frac{\alpha }{2}-\frac{\alpha }{2}\zeta \big)^{\frac{1}{\alpha}}(1-\zeta )\Big(\frac{1}{2}\zeta +\frac{\zeta }{1-\zeta }\Big),
	\end{equation}
	and $\beta(z)$ be defined in \eqref{2.2}.
	Then we have
	$$
	\big| \mu (e^{-\beta(z)\tau})-1 \big|\le C\left| z \right|^2\tau^2,\ z\in \Gamma_{\theta ,\kappa }^{\tau }.
	$$
	\end{lemma}
	\begin{proof}
    Let $\omega:=\beta(z)$.
	By Lemma \ref{l2.2}, it yields that
	$0<\left| \omega \tau  \right|\le C\left| z \right|\tau \le C\frac{\pi }{\sin \theta }$ for $z\in \Gamma_{\theta ,\kappa }^{\tau }$.
	Let $\gamma(\zeta)=(1-\zeta)^\alpha(1+\frac{\alpha}{2}-\frac{\alpha}{2}\zeta)$, then we have
	$$
	\begin{aligned}
	\mu (\zeta )-1 & =\gamma(\zeta)^{1/\alpha}\Big( \frac{\zeta }{1-\zeta }+\frac{1}{2}\zeta \Big)-1  \\
	{} & =\Big( 1+\frac{1}{2}(1-\zeta )+\frac{1-\alpha }{8}{{(1-\zeta )}^{2}}+\cdots \Big)\Big( 1-\frac{1}{2}(1-\zeta )-\frac{1}{2}{(1-\zeta )^2} \Big)-1  \\
	{} & =O\big( (1-\zeta)^2 \big)\text{ as }\zeta \to 1,  \\
	\end{aligned}
	$$
	which implies that
	$\mu ({{e}^{-\omega \tau }})-1=O({{(1-{{e}^{-\omega \tau }})}^{2}})=O({{(\omega \tau )}^{2}})$ with $\omega \tau \to 0$.
	Then, there exists ${{\delta }_{0}}>0$ with $0<{{\delta }_{0}}\le \frac{C\pi }{\sin \theta }$ such that $\left| \mu ({{e}^{-\omega \tau }})-1 \right|\le C{{\left| \omega \tau  \right|}^{2}}\le C{{\left| z\tau  \right|}^{2}}$ holds for $0\le \left| \omega \tau  \right|\le {{\delta }_{0}}$ by using Lemma \ref{l2.2}.
	
	For the case ${{\delta }_{0}}\le \left| \omega \tau  \right|\le \frac{C\pi }{\sin \theta }$,
	the function $\mu (\zeta )-1$ is continuous at any $\zeta \ne 1$, which implies that $\mu ({{e}^{-\omega \tau }})-1$ is continuous at any $\omega \tau\ne 0$.
	Hence,
	$\mu ({{e}^{-\omega \tau }})-1$ is bounded for ${{\delta }_{0}}\le \left| \omega \tau  \right|\le \frac{C\pi }{\sin \theta }$, $z\in \Gamma_{\theta ,\kappa }^{\tau }$, and
	$$
	\left| \mu ({{e}^{-\omega \tau }})-1 \right|\le C\text{=}C\delta_{0}^{-2}\delta_{0}^{2}\le C\delta_{0}^{-2}{{\left| \omega  \right|}^{2}}{{\tau }^{2}}\le C{{\left| z \right|}^{2}}{{\tau }^{2}}.
	$$
    Therefore, we obtain
	$$
	\left| \mu ({{e}^{-\omega \tau }})-1 \right|\le C{{\left| z \right|}^{2}}{{\tau }^{2}},~\forall~z\in \Gamma_{\theta ,\kappa }^{\tau }.
	$$
	\end{proof}
	
	\begin{lemma}\label{l2.7}
	Let $a_\tau(\zeta )=(1-\zeta)(1+\frac{\alpha}{2}-\frac{\alpha }{2}\zeta )^{\frac{1}{\alpha}}/\tau$ be defined in \eqref{atauzeta}, and $\beta(z)$ be defined in \eqref{2.2}. Then for real number $\sigma$, we have
	$$
	\big |a_\tau(e^{-\beta(z) \tau })^\sigma-\beta(z)^\sigma \big |\le C\tau^2|\beta(z)|^{2+\sigma},~~\forall~ z\in \Gamma_{\theta,\kappa}^\tau.
	$$
	\end{lemma}
	\begin{proof}
	Let $\gamma(\zeta)=(1-\zeta)^\alpha(1+\frac{\alpha}{2}-\frac{\alpha}{2}\zeta)$ and $\omega:=\beta(z)$, we can derive that
	$$
	\begin{aligned}
	a_\tau(e^{-\omega \tau})^\sigma-\omega^\sigma & =\tau^{-\sigma}\big(\gamma(e^{-\omega\tau})^{\frac{\sigma}{\alpha}}-(\omega\tau)^\sigma \big)=\frac{(\omega\tau )^\sigma\left( 1+d_1(\omega\tau)^2+\cdots  \right)^{\frac{\sigma}{\alpha }}-(\omega\tau)^\sigma }{\tau^\sigma }  \\
	&=\frac{(\omega\tau )^\sigma\left( 1+\frac{{\gamma d_1}}{\alpha }(\omega\tau )^2+\cdots\right)-(\omega\tau)^\sigma }{\tau^\sigma }=\tau^{-\sigma}O\left((\omega\tau )^{2+\sigma} \right),~ \text{as}\ \omega\tau \to 0.  \\
	\end{aligned}
	$$
	This implies that there exists $0<{{\delta }_{0}}\le \frac{C\pi }{\sin \theta }$ such that
	$$
	\left| {{a}_{\tau }}({{e}^{-\omega \tau }})^\sigma-\omega^\sigma\right|\le C{{\tau }^{2}}{{\left| \omega  \right|}^{2+\sigma}},~~\ 0\le \left| \omega \tau \right|\le {\delta_0},\ z\in \Gamma_{\theta ,\kappa }^{\tau }.
	$$
	For the case ${\delta_0}\le \left| \omega \tau  \right|\le \frac{C\pi }{\sin \theta }$ with $z\in \Gamma_{\theta,\kappa }^{\tau}$, the results in Lemmas \ref{l2.2} and \ref{l2.3} follow that
	$$
	\begin{aligned}
	\left| a_\tau(e^{-\omega \tau})^\gamma-\omega^\gamma \right|&\le \left| a_\tau(e^{-\omega \tau})^\gamma \right|+\left| \omega^\gamma  \right|\le C\left| \omega  \right|^\gamma\le C(\left| \omega  \right|^{2+\gamma}\tau^2)\frac{1}{\left| \omega  \right|^2\tau^2}\\
	&\le C(\left| \omega  \right|^{2+\gamma}\tau^2)\frac{1}{\delta_0^2}\le C\left| \omega  \right|^{2+\gamma}\tau^2.
	\end{aligned}
	$$
	Thus, the result is obtained.
	\end{proof}

%-----------------------------------------------------------------------------------------------------
	\section{A second-order WSGD scheme}\label{Section3}
%-----------------------------------------------------------------------------------------------------
	In this section, we propose a second-order WSGD scheme in time for solving the backward fractional Feynman-Kac equation \eqref{2.1} by using the weighted and shifted Gr\"{u}nwald formula to approximate the Riemann-Liouville fractional substantial derivative in time with some initial corrections.

	We first recall the weighted and shifted Gr\"{u}nwald formula proposed in \cite{TianZD:2015} for approximating the Riemann-Liouville fractional derivative \eqref{eq:defRLfd}, and then discuss its extension to the approximation of the Riemann-Liouville fractional substantial derivative \eqref{eq:defRLfsd} in time.

	Let $0 = t_0 < t_1 < \cdots<t_N = T$ be a temproal partition of $[0,T]$ and $\tau$ be the step size with grid points $t_n=n\tau,\ n=0,1,\cdots,N$. By choosing the shift index $(p,q)=(0,-1)$ in \cite{TianZD:2015}, the Riemann-Liouville fractional derivative $_{0}^{{}}D_{t}^{\alpha }G(t)$ at $t = t_n$ with $n \ge 1$ can be approximated by
	\begin{equation}\label{eq:WSGL}
	{ }_{0} D_{t}^{\alpha} G(t_{n}) \approx\frac{1}{\tau^{\alpha}}\sum_{j=0}^{n} w_{n-j}^{(\alpha)}G(t_{j})
	\end{equation}
	with second-order of accuracy, where
	$$
	w_{0}^{(\alpha)}=\frac{\alpha+2}{2} g_{0}^{(\alpha)},\quad w_{j}^{(\alpha)}=\frac{\alpha+2}{2} g_{j}^{(\alpha)}-\frac{\alpha}{2} g_{j-1}^{(\alpha)},\quad j=1,2, \ldots, n,
	$$
	and the coefficients $\{g_j^{(\alpha)},j = 0,1,2,\cdots\}$ satisfy
	$$
	\delta_{1}(\zeta)^{\alpha}=\Big(\frac{1-\zeta}{\tau}\Big)^{\alpha}=\frac{1}{\tau^{\alpha}}\sum_{j=0}^{\infty} g_{j}^{(\alpha)} \zeta^{j}.
	$$
	The explicit expression of $g_{j}^{(\alpha)}$ is
	$$
	g_{j}^{(\alpha)}=(-1)^{j}\binom{\alpha}{j}=(-1)^{j}\frac{\Gamma(\alpha+1)}{\Gamma(j+1)\Gamma(\alpha-j+1)}.
	$$
	The weights $\{w_j^{(\alpha)},j = 0, 1, 2,\cdots\}$ satisfy
	\begin{equation*}
	\frac{1}{\tau^{\alpha}}\sum_{j=0}^{\infty} w_{j}^{(\alpha)} \zeta^{j}
	=\frac{1}{\tau^{\alpha}}\frac{\alpha+2}{2}\sum_{j=0}^{\infty}g_{j}^{(\alpha)} \zeta^{j}-\frac{1}{\tau^{\alpha}}\frac{\alpha}{2} \zeta\sum_{j=0}^{\infty} g_{j}^{(\alpha)}\zeta^{j}
	=\frac{(1-\zeta)^{\alpha}\big(1+\frac{\alpha}{2}-\frac{\alpha}{2} \zeta\big)}{\tau^\alpha},
	\end{equation*}
	then the corresponding generating function of the weights $\{w_j^{(\alpha)},j = 0, 1, 2,\cdots\}$ is $a_\tau(\zeta)^\alpha$ with $a_\tau(\zeta)$ denoted by
	\begin{equation}\label{atauzeta}
	a_\tau(\zeta)=\frac{(1-\zeta)\big(1+\frac{\alpha}{2}-\frac{\alpha}{2} \zeta\big)^{1/\alpha}}{\tau}.
	\end{equation}
	
	From the definition of the Riemann-Liouville fractional substantial derivative in \eqref{eq:defRLfsd}, its approximation by the formula \eqref{eq:WSGL} can be extended as follows
	$$
    {}_0D_t^{\alpha,x}G(t_n)\approx\frac{1}{\tau^{\alpha}}e^{-t_n\rho U(x)}\sum\limits_{j=0}^{n}w_j^{(\alpha)}e^{t_{n-j}\rho U(x)}G^{n-j}=\frac{1}{\tau^{\alpha}}\sum\limits_{j=0}^{n}w_j^{(\alpha)}e^{-t_{j}\rho U(x)}G^{n-j}.
	$$
	Similarly, the approximation of the Riemann-Liouville fractional substantial integral in \eqref{2.1} is
	$${}_{0}I_{t}^{1-\alpha,x}\phi(t_n)\approx\frac{1}{\tau^{\alpha-1}}\sum\limits_{j=0}^{n}w_{j}^{(\alpha-1)}e^{-t_j\rho U(x)}\phi^{n-j},
	$$
	where the coefficients $w_{j}^{(\alpha-1)}$ satisfy $\big(a_\tau(\zeta)\big)^{\alpha-1}=\frac{1}{\tau^{\alpha-1}}\sum_{j=0}^{\infty}w_j^{(\alpha-1)}\zeta^j$. \\
	
	Using the relationship ${}_0^CD_t^{\alpha,x}G(t)={}_0D_t^{\alpha,x}\big(G(t)-e^{-t\rho U(x)}G_0\big)$ and the above approximate formulae,
	we can obtain a temporal discrete scheme of the backward fractional Feynman-Kac equation \eqref{2.1} as follows
	\begin{equation}\label{3.5}
	\begin{aligned}
	&\frac{1}{\tau^{\alpha}}\sum\limits_{j=0}^{n-1}w_{j}^{(\alpha)}e^{-t_j\rho U(x)}G^{n-j}-\frac{1}{\tau^{\alpha}}\sum\limits_{j=0}^{n-1}w_{j}^{(\alpha)}e^{-t_n\rho U(x)}G_0+AG^n\\
	&=\frac{1}{\tau^{\alpha-1}}\sum\limits_{j=0}^{n}w_{j}^{(\alpha-1)}e^{-t_j\rho U(x)}f^{n-j},
	\end{aligned}
	\end{equation}
	where $f^n=f(t_n)$. As indicated in literatures \cite{JinLZ:2017,SunND:2020,SunND:2021a,WangYYP:2020}, the direct application of BDFs and the formula \eqref{eq:WSGL} for time-fractional equations can not achieves their optimal convergence orders, and some corrections should be meticulously designed, and this also happens to \eqref{3.5}.
	
	In order to capture the regularity property of the solution of \eqref{2.1} at $t=0$ and preserve the optimal second-order convergence rate, we correct the scheme \eqref{3.5} as follows
	\begin{equation}\label{3.6}
	\begin{aligned}
	&\frac{1}{\tau^{\alpha}}\sum\limits_{j=0}^{n-1}w_{j}^{(\alpha)}e^{-t_j\rho U(x)}G^{n-j}-\frac{1}{\tau^{\alpha}}\sum\limits_{j=0}^{n-1}w_{j}^{(\alpha)}e^{-t_n\rho U(x)}G_0+AG^n\\
	&=\frac{1}{2\tau^{\alpha}}w_{n-1}^{(\alpha)}e^{-t_n\rho U(x)}G_0+\frac{1}{\tau^{\alpha-1}}\sum\limits_{j=0}^{n-1}w_{j}^{(\alpha-1)}e^{-t_j\rho U(x)}f^{n-j}\\
	&~~~~+\frac{1}{2\tau^{\alpha-1}}w_{n-1}^{(\alpha-1)}e^{-t_{n-1}\rho U(x)}f^0,
	\end{aligned}
	\end{equation}
	for $n=1,2\cdots N.$
	
%-----------------------------------------------------------------------------------------------------
	\section{Error estimates}\label{Section4}
%-----------------------------------------------------------------------------------------------------
	In this section, we analyze the temporal error estimates for the WSGD scheme \eqref{3.6} for the homogeneous and inhomogeneous problems. The results depend only on the regularity assumptions on the data, without any regularity requirements on the solution of the equation.
	
	\subsection{Homogeneous case}\label{Section4.2}
	We first analyze the homogeneous case for the scheme \eqref{3.6}, i.e., $f=0$.
	Multiplying $\zeta^n$ on both sides of \eqref{3.6} and summing $n$ from $1$ to $\infty$ lead to
	\begin{equation*}
	\begin{aligned}
	&\frac{1}{\tau^\alpha}\sum\limits_{n=1}^{\infty}\sum\limits_{j=0}^{n-1}w_j^{(\alpha)} e^{-t_j\rho U(x)}G^{n-j}\zeta^n+\sum\limits_{n=1}^{\infty}AG^n\zeta^n\\
	&=\frac{1}{\tau^\alpha}\sum\limits_{n=1}^{\infty}\sum\limits_{j=0}^{n-1}w_j^{(\alpha)} e^{-t_n\rho U(x)}G_0\zeta^n+\frac{1}{2\tau^\alpha}\sum\limits_{n=1}^{\infty}w_{n-1}^{(\alpha)} e^{-t_n\rho U(x)}\zeta^nG_0.
	\end{aligned}
	\end{equation*}
	Recall that $a_\tau(\zeta)=(1-\zeta)(1+\frac{\alpha }{2}-\frac{\alpha}{2}\zeta)^{\frac{1}{\alpha}}/\tau$ in \eqref{atauzeta}. Then it follows that
	$$
	\big(a_\tau(e^{-\tau \rho U(x)}\zeta )^\alpha +A\big)\sum\limits_{n=1}^{\infty}G^n\zeta^n=a_\tau(e^{-\tau \rho U(x)}\zeta )^\alpha\Big(\frac{e^{-\tau \rho U(x)}\zeta }{1-e^{-\tau \rho U(x)}\zeta }+\frac{1}{2}e^{-\tau \rho U(x)}\zeta\Big)G_0,
	$$
	and we obtain
	$$
	\sum\limits_{n=1}^{\infty}G^n\zeta^n=\big(a_\tau(e^{-\tau \rho U(x)}\zeta )^\alpha +A\big)^{-1}a_\tau(e^{-\tau \rho U(x)}\zeta )^\alpha\Big(\frac{e^{-\tau \rho U(x)}\zeta }{1-e^{-\tau \rho U(x)}\zeta }+\frac{1}{2}e^{-\tau \rho U(x)}\zeta\Big)G_0.
	$$
	By Cauchy's integral formula and the definition of $\mu(\zeta)$ in \eqref{muzeta}, it holds that
	\begin{equation*}
	G^n=\frac{1}{2\pi \mathbf{i}}\int_{|\zeta|=\xi_\tau}\zeta^{-n-1}\big(a_\tau(e^{-\tau \rho U(x)}\zeta )^\alpha +A\big)^{-1}a_\tau(e^{-\tau \rho U(x)}\zeta )^{\alpha-1}\mu(e^{-\tau\rho U(x)}\zeta)\tau^{-1}G_0 d\zeta,
	\end{equation*}
	where $\xi_\tau=e^{-\tau(\kappa +1)}$. Let $\zeta=e^{-z\tau}$ and $\beta(z)$ be defined in \eqref{2.2}, we arrive at
	\begin{equation*}
	G^n=\frac{1}{2\pi \mathbf{i}}\int_{\Gamma^{\tau}}e^{zt_n}\big(a_\tau(e^{- \beta(z)\tau})^\alpha +A\big)^{-1}a_\tau(e^{- \beta(z)\tau} )^{\alpha-1}\mu(e^{-\beta(z)\tau})G_0dz,
	\end{equation*}
	where $\Gamma^\tau=\{z=\kappa+1+\mathbf{i} y:\ y \in \mathbb{R}\ \text {and}\ |y| \leq \pi /\tau\}$.
	By deforming the contour $\Gamma^\tau$ to $\Gamma_{\theta,\kappa}^\tau$, it implies
	\begin{equation}\label{4.6}
	G^n=\frac{1}{2\pi \mathbf{i}}\int_{\Gamma_{\theta,\kappa}^{\tau}}e^{zt_n}\big(a_\tau(e^{- \beta(z)\tau})^\alpha +A\big)^{-1}a_\tau(e^{- \beta(z)\tau} )^{\alpha-1}\mu(e^{-\beta(z)\tau})G_0dz.
	\end{equation}
	
	Before analyzing the error estimate, we provide the following lemma.
	\begin{lemma}\label{l4.2}
	Let $\beta(z)$, $\mu(\zeta)$ and $a_\tau(\zeta)$ be defined in \eqref{2.2}, \eqref{muzeta} and \eqref{atauzeta}, respectively, then it holds that
	$$
	\big\|\big(\beta (z)^\alpha+A \big)^{-1}\beta (z)^{\alpha-1}-\big(a_\tau(e^{-\beta(z) \tau })^\alpha +A\big)^{-1}a_\tau(e^{-\beta(z) \tau})^{\alpha-1}\mu(e^{- \beta(z)\tau})  \big\|\le C\tau^2\left| z \right|,\  z\in \Gamma_{\theta,\kappa}^\tau.
	$$
	\end{lemma}
	\begin{proof}
	First, by the triangle inequality, it follows that
	$$
	\big\|\big(\beta (z)^\alpha+A\big)^{-1}\beta (z)^{\alpha-1}-\big(a_\tau(e^{-\beta(z) \tau })^\alpha +A\big)^{-1}a_\tau(e^{-\beta(z)\tau })^{\alpha-1}\mu(e^{- \beta(z)\tau})  \big\|\le I+II,
	$$
	where
	$$
	I=\big\|\big( (\beta (z)^\alpha+A )^{-1}-( a_\tau(e^{-\beta(z)\tau })^\alpha+A )^{-1}\big) \beta(z)^{\alpha-1}\big\|,
	$$
	and
	$$
	II=\big\|\big( a_\tau(e^{- \beta(z)\tau})^\alpha+A\big)^{-1}\big( \beta(z)^{\alpha-1}-a_\tau(e^{-\beta(z) \tau})^{\alpha-1}\mu(e^{-\beta(z)\tau})\big) \big\|.
	$$
	For $I$, by Lemmas \ref{l2.2}, \ref{l2.5} and \ref{l2.7}, we have
	$$
	I\le\big\|\big(\beta (z)^\alpha+A\big)^{-1}\big\|\cdot|a_\tau(e^{- \beta(z)\tau})^\alpha-\beta(z)^{\alpha}|\cdot\big\|\big(a_\tau(e^{- \beta(z)\tau})^\alpha+A\big)^{-1}\big\|\cdot|\beta(z)^{\alpha-1}|\le C\tau^2|z|.
	$$
	Similarly, for $II$, we can get the following inequality by Lemma \ref{l2.5},
	$$
	II\le C|z|^{-\alpha}(II_1+II_2),
	$$
	where $II_1=\| \beta(z)^{\alpha-1}-a_\tau(e^{-\beta(z)\tau})^{\alpha-1}\|$ and $II_2=\|a_\tau(e^{- \beta(z)\tau})^{\alpha-1}\big(1-\mu(e^{-\beta(z)\tau})\big)\|$.
    Additionally, Lemmas \ref{l2.3}, \ref{l2.6} and \ref{l2.7} imply that
	$$
	II_1\le C\tau^2|z|^{\alpha+1},\quad II_2\le C\tau^2|z|^{\alpha+1}.
	$$
	Consequently, it yields that $II\le C\tau^2|z|$. In summary, we obtain $I+II\le C\tau^2|z|$, which completes the proof of the lemma.
	\end{proof}
	
	By using Lemma \ref{l4.2}, we can derive the error estimate of the WSGD scheme \eqref{3.6} for the  homogeneous problem in the following theorem.
	
	\begin{theorem}\label{t4.2}
	Let $U(x)$ be bounded in $\bar{\Omega}$, $G(t_n)$ and $G^n$ be the solutions of \eqref{2.1} and \eqref{3.6} respectively, and $\tau\leq \tau_*$ with $\tau_*$ being a positive constant. For $f=0$ and $G_0\in L^2(\Omega)$, we have
	$$\| G( t_n )-G^n\|\le C\tau^2t_n^{-2}\| G_0\|.$$
	\end{theorem}
	\begin{proof}
	Subtracting (\ref{4.6}) from \eqref{4.2} implies
	$$\begin{aligned}
	\left\| G(t_n)-G^n\right\|
	& \le C\Big\| \int_{\Gamma_{\theta,\kappa}} e^{zt_n}( \beta(z)^{\alpha}+A)^{-1}\beta(z)^{\alpha-1}G_0dz\\
	&\qquad -\int_{\Gamma_{\theta,\kappa}^\tau} e^{zt_n}\big( (a_\tau(e^{-\beta(z)\tau}))^{\alpha}+A\big)^{-1}a_\tau(e^{- \beta(z)\tau} )^{\alpha-1}\mu(e^{-\beta(z)\tau})G_0dz\Big\| \\
	& \le C\Big\|\int_{\Gamma_{\theta,\kappa}\backslash\Gamma_{\theta,\kappa}^\tau}
	e^{zt_n}(\beta(z)^{\alpha}+A)^{-1}\beta (z)^{\alpha-1}dz\Big\|\|G_0\| \\
	& \quad+C\Big\|\int_{\Gamma_{\theta,\kappa}^\tau}e^{zt_n}\big((\beta (z)^\alpha+A )^{-1}\beta (z)^{\alpha-1}\\
	&\qquad\qquad -(a_\tau(e^{-\beta(z) \tau })^{\alpha}+A )^{-1}a_\tau(e^{- \beta(z)\tau} )^{\alpha-1}\mu(e^{-\beta(z)\tau}) \big)dz \Big\|\|G_0\| \\
	& \le C(I+II)\| G_0 \|,
	\end{aligned}$$
	where
	$$
	\Gamma_{\theta,\kappa}\backslash\Gamma_{\theta,\kappa}^\tau=\big\{ z\in \mathbb{C}:z=re^{\pm \mathbf{i}\theta },~\frac{\pi }{\tau \sin\theta}\le r<\infty \big\},~~\theta\in\big(\pi/2,\pi\big).
	$$
	
	For $I$, we have from Lemma \ref{l2.2} that
	$$
	I\le \int_{\Gamma_{\theta,\kappa}\backslash\Gamma_{\theta,\kappa}^\tau}\left| e^{zt_n} \right|\left| z \right|^{-1}\left| dz \right|\le \int_{\frac{\pi }{\tau \sin\theta}}^{\infty}e^{rt_n\cos\theta}r^{-1}dr\le \int_{\frac{\pi }{\tau \sin\theta}}^{\infty}r^{-2}e^{rt_n\cos\theta}rdr\le C\tau^2t_n^{-2}.
	$$
	where $r^{-2}\le \frac{\tau^2\sin^2\theta}{\pi^2}\le C\tau^2$ as $r\ge\frac{\pi }{\tau \sin \theta }$.
	For $II$, combing with Lemma \ref{l4.2}, it holds that
	$$
	\begin{aligned}
	II&\le\int_{\Gamma_{\theta,\kappa}^\tau}|e^{zt_n}|\big\|(\beta (z)^\alpha+A )^{-1}\beta (z)^{\alpha-1}-(a_\tau(e^{-\beta(z) \tau })^{\alpha}+A )^{-1}a_\tau(e^{- \beta(z)\tau} )^{\alpha-1}\mu(e^{-\beta(z)\tau})\big\||dz|\\
	& \le C\tau^2\int_{\Gamma_{\theta ,\kappa }^{\tau }}{| e^{zt_n} |}\left| z \right|\left| dz \right|
	  \le C\tau^2\Big( \int_{\kappa }^{\frac{\pi }{\tau \sin\theta}}e^{rt_n\cos\theta}rdr+\int_{-\theta }^{\theta }e^{\kappa t_n \cos\varphi}\kappa^2d\varphi\Big) \\
	& \le C\tau^2t_n^{-2}.
	\end{aligned}
	$$
	This completes the proof of the theorem.
	\end{proof}
	
	\subsection{Inhomogeneous case}
	Now we turn to the inhomogeneous problem $f\ne 0$ with zero initial data $G_0=0$. In such case, we have from \eqref{2.3} that
	\begin{equation}\label{4.7}
	\widetilde{G}(z)=\big(\beta(z)^\alpha+A\big)^{-1}\beta(z)^{\alpha-1}\tilde{f}(z).
	\end{equation}
	Inspired by \cite{JinLZ:2018a,SunND:2021a}, we consider the Taylor expansion $f(t)=f(0)+tf'(0)+\big(t*f''(t)\big)(t)$, and estimate the error by two steps.
	
	We first consider the case of $f(t)=f(0)+tf'(0)$, then the equation \eqref{4.7} becomes
	\begin{equation*}
	\widetilde{G}(z)=\big(\beta(z)^\alpha+A\big)^{-1}\beta(z)^{\alpha-1}\left(z^{-1}f(0)+z^{-2}f^{\prime}(0)\right).
	\end{equation*}
	Taking the inverse Laplace transform on both sides derives that
	\begin{equation}\label{4.8}
	G(t_n)=\frac{1}{2\pi \textbf{i}}\int_{\Gamma_{\theta ,\kappa }}e^{zt_n}\big(\beta (z)^\alpha+A\big)^{-1}\beta (z)^{\alpha-1}\left(z^{-1}f(0)+z^{-2}f^{\prime}(0)\right)dz.
	\end{equation}
	For $f(t)=f(0)+tf'(0)$ and $G_0=0$, multiplying $\zeta^n$ on both sides of \eqref{3.6} and summing up with simple calculations, we can obtain
	$$
	\sum\limits_{n=1}^{\infty}G^n\zeta^n=\big(a_\tau(e^{-\tau \rho U(x)}\zeta )^\alpha +A\big)^{-1}a_\tau(e^{-\tau \rho U(x)}\zeta )^{\alpha-1}\Big(\big(\frac{\zeta }{1-\zeta }+\frac{1}{2}\zeta\big)f(0)+\frac{\tau\zeta}{(1-\zeta)^2}f'(0)\Big).
	$$
	The solution can be represented as follows by Cauchy's integral formula
	\begin{equation}\label{4.10}
	\begin{aligned}
	G^n=\frac{1}{2\pi \mathbf{i}}\int_{\Gamma_{\theta,\kappa }^{\tau}}&e^{zt_n}(a_\tau(e^{- \beta(z)\tau})^\alpha +A)^{-1}a_\tau(e^{-\beta(z)\tau} )^{\alpha-1}\\
	&\Big(a_\tau(e^{- z\tau} )^{-1}\mu(e^{-z\tau })f(0)+\frac{\tau^2e^{-z\tau}}{(1-e^{-z\tau})^2}f'(0)\Big)dz,
	\end{aligned}
	\end{equation}
	where $\mu(\zeta)$ is defined by \eqref{muzeta}.
	
	\begin{lemma}\label{l4.3}
	Let
	$$
	\begin{aligned}
	&\mathcal{A}(z)=
	\big(\beta (z)^\alpha+A\big)^{-1}\beta (z)^{\alpha-1}z^{-1}-\big(a_\tau(e^{- \beta(z)\tau})^\alpha +A\big)^{-1}a_\tau(e^{- \beta(z)\tau} )^{\alpha-1}a_\tau(e^{-z\tau } )^{-1}\mu(e^{- z\tau}),\\
	&\mathcal{B}(z)=
	\big(\beta (z)^\alpha+A\big)^{-1}\beta (z)^{\alpha-1}z^{-2}-\big(a_\tau(e^{- \beta(z)\tau})^\alpha +A\big)^{-1}a_\tau(e^{- \beta(z)\tau})^{\alpha-1}\frac{\tau^2e^{-z\tau}}{(1-e^{-z\tau})^2},
	\end{aligned}
	$$
	where $\beta(z)$, $\mu(\zeta)$ and $a_\tau(\zeta)$ are respectively defined in \eqref{2.2}, \eqref{muzeta} and \eqref{atauzeta}. Then we have
	\begin{itemize}
	  \item[(i)] $\|\mathcal{A}(z)\|\le C\tau^2,~~\forall~ z\in \Gamma_{\theta,\kappa}^\tau$,
	  \item[(ii)] $\|\mathcal{B}(z)\|\le C\tau^2|z|^{-1},~~\forall~ z\in \Gamma_{\theta,\kappa}^\tau$.
	\end{itemize}
	\end{lemma}
	\begin{proof}
	(i) It has from the formula of $\mathcal{A}(z)$ that
	$$
	\|\mathcal{A}(z)\|\le I+II,
	$$
	where
	$$
	I=\big\|  \big(( \beta (z)^\alpha+A )^{-1}-(a_\tau(e^{- \beta(z)\tau})^\alpha+A )^{-1}\big) \beta(z)^{\alpha-1}z^{-1}\big\|,
	$$
	and
	$$
	II=\big\| ( a_\tau(e^{- \beta(z)\tau})^\alpha+A )^{-1}\big( \beta(z)^{\alpha-1}z^{-1}-a_\tau(e^{-\beta(z) \tau})^{\alpha-1}a_\tau(e^{-z\tau})^{-1}\mu(e^{-z\tau})\big) \big\|.
	$$
	For $I$, by Lemmas \ref{l2.2}, \ref{l2.5} and \ref{l2.7}, we have
	$$
	I\le C\tau^2.
	$$
	Similarly, for $II$, an application of Lemma \ref{l2.5} implies
	$$
	II\le C|z|^{-\alpha}(II_1+II_2+II_3),
	$$
	where $II_1=\|\beta(z)^{\alpha-1}\big(z^{-1}-a_\tau(e^{-z\tau})^{-1}\big)\|$, $II_2=\| \big(\beta(z)^{\alpha-1}-a_\tau(e^{-\beta(z)\tau})^{\alpha-1}\big)a_\tau(e^{-z\tau})^{-1} \|$ and $II_3=\|a_\tau(e^{-\beta(z) \tau})^{\alpha-1}a_\tau(e^{-z\tau})^{-1}\big(1-\mu(e^{-z\tau})\big)\|$.
	Then we can obtain
	$$
	II_1\le C\tau^2|z|^{\alpha},~~II_2\le C\tau^2|z|^{\alpha},~~II_3\le C\tau^2|z|^{\alpha}
	$$
	by Lemmas \ref{l2.2}, \ref{l2.3}, \ref{l2.6} and \ref{l2.7}.
	Consequently, $II\le C\tau^2$. In summary,  $\|\mathcal{A}(z)\|\le C\tau^2$.
	
	(ii) It easily follows that $\mathcal{B}(z)$ satisfies
	$$
	\|\mathcal{B}(z)\|\le III+IV,
	$$	
	where
	$$
	III=\big\|  \big( ( \beta (z)^\alpha+A )^{-1}-( a_\tau(e^{-\beta(z)\tau })^\alpha+A )^{-1}\big) \beta(z)^{\alpha-1}z^{-2}\big\|,
	$$
	and
	$$
	IV=\Big\| \big( a_\tau(e^{- \beta(z)\tau})^\alpha+A \big)^{-1}\Big( \beta(z)^{\alpha-1}z^{-2}-a_\tau(e^{-\beta(z) \tau})^{\alpha-1}\frac{\tau^2e^{-z\tau}}{(1-e^{-z\tau})^2}\Big) \Big\|.
	$$
	For $III$, by Lemmas \ref{l2.2}, \ref{l2.5} and \ref{l2.7}, we have
	$$
	III\le C\tau^2|z|^{-1}.
	$$
	For $IV$, it yields from Lemma \ref{l2.5} that
	$$
	IV\le C|z|^{-\alpha}(IV_1+IV_2).
	$$
	where $IV_1=\|\big(\beta(z)^{\alpha-1}-a_\tau(e^{-\beta(z)\tau})^{\alpha-1}\big)z^{-2}\|$ and $IV_2=\big\| a_\tau(e^{- \beta(z)\tau})^{\alpha-1}\big(z^{-2}-\frac{\tau^2e^{-z\tau}}{(1-e^{-z\tau})^2}\big)\big\|$.
	We can also obtain
	$$
	IV_1\le C\tau^2|z|^{\alpha-1},~~IV_2\le C\tau^2|z|^{\alpha-1}.
	$$
	by applying Lemmas \ref{l2.3}, \ref{l2.7} and the estimate $\big|z^{-2}-\frac{\tau^2e^{-z\tau}}{(1-e^{-z\tau})^2}\big|\le C\tau^2$ in \cite[Lemma 3.4]{GunzburgerW:2019b}.
	Hence, it holds $IV\le C\tau^2|z|^{-1}$ which implies that $\|\mathcal{B}(z)\|\le III+IV\le C\tau^2|z|^{-1}$.
	This completes the proof.
	\end{proof}
	
	By the above lemma, we can obtain the error estimate for the case $f(t)=f(0)+tf'(0)$ and $G_0=0$ in the following theorem, its proof is analogous to the approach for Theorem \ref{t4.2}.
	
	\begin{theorem}\label{t4.3}
	Let $U(x)$ be bounded in $\bar{\Omega}$, $G(t_n)$ and $G^n$ be the solutions of \eqref{2.1} and \eqref{3.6}, respectively, and $\tau\leq \tau_*$ with $\tau_*$ being a positive constant. If $f(t)=f(0)+tf'(0)$ and $G_0=0$, then we have
	$$
	  \| G(t_n)-G^n \|\le C\tau^2\big(t_n^{-1}\| f(0) \|+\| f'(0) \|\big).
	$$
	\end{theorem}
	\begin{proof}
	Subtracting (\ref{4.10}) from (\ref{4.8}), it has
	$$
	\begin{aligned}
	&\left\|G(t_n)-G^n\right\|\\
	&\le C\Big(\Big\|\int_{\Gamma_{\theta,\kappa}\backslash\Gamma_{\theta,\kappa}^\tau}e^{zt_n}\big(\beta (z)^\alpha+A\big)^{-1}\beta (z)^{\alpha-1}z^{-1}dz\Big\|+\Big\|\int_{\Gamma_{\theta,\kappa}^\tau}e^{zt_n}\mathcal{A}(z) dz\Big\|\Big)\|f(0)\|\\
	&~~~+C\Big(\Big\|\int_{\Gamma_{\theta,\kappa}\backslash\Gamma_{\theta,\kappa}^\tau}e^{zt_n}\big(\beta (z)^\alpha+A\big)^{-1}\beta (z)^{\alpha-1}z^{-2}dz\Big\|+\Big\|\int_{\Gamma_{\theta,\kappa}^\tau}e^{zt_n}\mathcal{B} (z)dz\Big\|\Big)\|f'(0)\|\\
	&:= C(I+II)\|f(0)\|+C(III+IV)\|f'(0)\|,
	\end{aligned}
	$$
	where $\mathcal{A}(z)$ and $\mathcal{B}(z)$ are defined in Lemma \ref{l4.3}.
	
	By using Lemma \ref{l2.2}, we have the estimate for $I$ and $III$ that
	$$
	I\le \int_{\Gamma_{\theta,\kappa}\backslash\Gamma_{\theta,\kappa}^\tau}\left| e^{zt_n} \right|\left| z \right|^{-2}\left| dz \right| \le \int_{\frac{\pi }{\tau \sin\theta}}^{\infty}r^{-2}e^{rt_n\cos\theta}dr\le C\tau^2t_n^{-1},
	$$
	$$
	III\le \int_{\Gamma_{\theta,\kappa}\backslash\Gamma_{\theta,\kappa}^\tau}\left| e^{zt_n} \right|\left| z \right|^{-3}\left| dz \right| \le \int_{\frac{\pi }{\tau \sin\theta}}^{\infty}r^{-2}e^{rt_n\cos\theta}r^{-1}dr\le C\tau^2.
	$$
	For $II$ and $IV$, it obtains from Lemma \ref{l4.3} that
	$$
	II\le C\tau^2\int_{\Gamma_{\theta,\kappa}^\tau}|e^{zt_n}| |dz|\le C\tau^2t_n^{-1},
	$$
	$$
	IV\le C\tau^2\int_{\Gamma_{\theta,\kappa}^\tau}|e^{zt_n}||z|^{-1} |dz|\le C\tau^2.
	$$
	Therefore, the estimate is derived from the above arguments.
	\end{proof}
	
	Next we analyze the error estimate for the case $f(t)=\big(t*f''(t)\big)(t)$ with $f(0)=0$ and $G_0=0$. From \eqref{2.3} and \eqref{4.2}, the corresponding solution can be represented by
	\begin{equation}\label{4.13}
	G(t_n)=\big(\mathscr{E}(t)*f(t)\big)(t_n)=\big(\mathscr{E}(t)*t*f''(t)\big)(t_n)=\big((\mathscr{E}(t)*t)*f''(t)\big)(t_n),
	\end{equation}
	where
	\begin{equation}\label{Et}
	\mathscr{E}(t)=\frac{1}{2\pi\mathbf{i}}\int_{\Gamma_{\theta,\kappa}}e^{zt}\big(\beta(z)^\alpha+A\big)^{-1}\beta(z)^{\alpha-1}dz.
	\end{equation}
	
	For $f(0)=0$ and $G_0=0$, \eqref{3.6} becomes
	$$
	\begin{aligned}
	\frac{1}{\tau^{\alpha}}\sum\limits_{j=0}^{n-1}w_{j}^{(\alpha)}e^{-t_j\rho U(x)}G^{n-j}+AG^n&=\frac{1}{\tau^{\alpha-1}}\sum\limits_{j=0}^{n-1}w_{j}^{(\alpha-1)}e^{-t_j\rho U(x)}f^{n-j}.
	\end{aligned}
	$$
	Then multiplying $\zeta^n$ on both sides and summing up from $1$ to $\infty$, we can obtain
	\begin{equation}\label{gnzetan}
	\sum\limits_{n=1}^{\infty}G^n\zeta^n=\big(a_\tau(e^{-\tau \rho U(x)}\zeta )^\alpha +A\big)^{-1}a_\tau(e^{-\tau \rho U(x)}\zeta )^{\alpha-1}\sum_{n=1}^{\infty}f(t_n)\zeta^n.
	\end{equation}
	By Cauchy’s integral formula and Cauchy’s integral theorem, it obtains that 
	\begin{equation}\label{etaunzetan}
	\big(a_\tau(e^{-\tau \rho U(x)}\zeta )^\alpha +A\big)^{-1}a_\tau(e^{-\tau \rho U(x)}\zeta )^{\alpha-1}=\sum_{n=0}^{\infty}\mathscr{E}_\tau^n\zeta^n,
	\end{equation}
	where
	$$
	\mathscr{E}_{\tau}^n=\frac{\tau}{2\pi\mathbf{i}}\int_{\Gamma_{\theta,\kappa}^\tau}e^{zt_n}\big(a_\tau(e^{-\tau \beta(z)})^\alpha+A\big)^{-1}a_\tau(e^{-\tau \beta(z)})^{\alpha-1}dz.
	$$
	Thus \eqref{gnzetan} and \eqref{etaunzetan} imply that
	$$
	G^n=\sum_{j=1}^{n}\mathscr{E}_{\tau}^{n-j}f(t_j)=\big(\mathscr{E}_{\tau}(t)*f(t)\big)(t_n),
	$$
	where
	\begin{equation}\label{etaut}
	\mathscr{E}_{\tau}(t)=\sum_{j=0}^{\infty}\mathscr{E}_{\tau}^j\delta_{t_j}(t).
	\end{equation}
	Moreover, the discrete solution of \eqref{3.6} with $f(t)=(t*f'')(t)$ and $G_0=0$ can be formulated by
	\begin{equation}\label{4.14}
	G^n=(\mathscr{E}_{\tau}*f)(t_n)=(\mathscr{E}_{\tau}*t*f'')(t_n)=\big((\mathscr{E}_{\tau}*t)*f''\big)(t_n).
	\end{equation}

	\begin{theorem}\label{t4.5}
	Let $U(x)$ be bounded in $\bar{\Omega}$, $G(t_n)$ and $G^n$ be the solutions of \eqref{2.1} and \eqref{3.6} respectively, and $\tau\leq \tau_*$ with $\tau_*$ being a positive constant. If $f(t)=(t*f'')(t)$ and $G_0=0$, then we have
	$$\left\| G(t_n)-G^n \right\|\le C\tau^2\int_{0}^{t_n}\| f''(s)\|ds.$$
	\end{theorem}
	\begin{proof}
    By the definition of $\mathscr{E}_{\tau}(t)$ in \eqref{etaut}, we have
    $$
      \left(\mathscr{E}_\tau * t\right)\left(t_n\right)=\sum_{j=1}^n \mathscr{E}_\tau^{n-j} t_j.
    $$ 
    Then, with \eqref{etaunzetan}, we can obtain that
    $$
    \begin{aligned}
    \sum_{n=1}^{\infty}\left(\mathscr{E}_\tau * t\right)\left(t_n\right) \zeta^n=\sum_{n=1}^{\infty} \sum_{j=1}^n \mathscr{E}_\tau^{n-j} t_j \zeta^n=\left(\sum_{n=0}^{\infty} \mathscr{E}_\tau^n \zeta^n\right)\left(\sum_{n=1}^{\infty} t_n \zeta^n\right)\\
    =\big(a_\tau(e^{-\tau \rho U(x)}\zeta )^\alpha +A\big)^{-1}a_\tau(e^{-\tau \rho U(x)}\zeta )^{\alpha-1}\frac{\tau\zeta}{(1-\zeta)^2}.
    \end{aligned}
    $$
    By Cauchy’s integral formula and Cauchy’s integral theorem, we can derive that
    $$
    \begin{aligned}
      \left(\mathscr{E}_\tau * t\right)\left(t_n\right)
      =\frac{1}{2 \pi \mathbf{i}} \int_{\Gamma_{\theta, \kappa}^\tau} e^{z t_n} \big(a_\tau(e^{- \beta(z)\tau})^\alpha +A\big)^{-1}a_\tau(e^{- \beta(z)\tau})^{\alpha-1}\frac{\tau^2e^{-z\tau}}{(1-e^{-z\tau})^2}dz.
    \end{aligned}
    $$
    It obtains from \eqref{Et} that 
    $$
    (\mathscr{E}*t)\left(t_n\right)=\frac{1}{2 \pi \mathbf{i}} \int_{\Gamma_{\theta, \kappa}} e^{z t_n} \big(\beta (z)^\alpha+A\big)^{-1}\beta (z)^{\alpha-1}z^{-2}dz.
    $$
	Hence, the proof for the term $f(t)=tf'(0)$ in Theorem \ref{t4.3} shows that
    \begin{equation*}
      \big\|\big((\mathscr{E}-\mathscr{E}_{\tau})*t\big)(t_n)\big\| \le C\tau^2.
    \end{equation*}
    
    Similarly, we can also derive that
    $$
    \begin{aligned}
      &\left(\mathscr{E}_\tau * 1\right)\left(t_n\right)
      =\frac{1}{2 \pi \mathbf{i}} \int_{\Gamma_{\theta, \kappa}^\tau} e^{z t_n} \big(a_\tau(e^{- \beta(z)\tau})^\alpha +A\big)^{-1}a_\tau(e^{- \beta(z)\tau})^{\alpha-1}\frac{\tau e^{-z\tau}}{1-e^{-z\tau}}dz,\\
      &(\mathscr{E}*1)\left(t_n\right)=\frac{1}{2 \pi \mathbf{i}} \int_{\Gamma_{\theta, \kappa}} e^{z t_n} \big(\beta (z)^\alpha+A\big)^{-1}\beta (z)^{\alpha-1}z^{-1}dz.
    \end{aligned}
    $$
    As $|\frac{\tau e^{-z\tau}}{1-e^{-z\tau}}-z^{-1}|\le C\tau$ holds by \cite[(C.1)]{GunzburgerLW:2019b}, it leads to the following estimate
    \begin{equation*}
      \big\|\big((\mathscr{E}-\mathscr{E}_{\tau})*1\big)(t_n)\big\| \le C\tau.
    \end{equation*}

    Taking Taylor expansion of $\big((\mathscr{E}-\mathscr{E}_{\tau})*t\big)(t)$ at $t=t_n$, then it implies that
    \begin{equation*}
    \begin{aligned}
      \big\|\big((\mathscr{E}-\mathscr{E}_{\tau})*t\big)(t)\big\|
      &\le\big\|\big((\mathscr{E}-\mathscr{E}_{\tau})*t\big)(t_n)\big\|+|t-t_n|\big\|\big((\mathscr{E}-\mathscr{E}_{\tau})*1\big)(t_n)\big\|\\
      &~~~~+\big\|\int_{t_n}^t(t-s)(\mathscr{E}-\mathscr{E}_{\tau})(s)ds\big\|\\
      &\le C\tau^2,~~\forall~t\in(t_{n-1},t_n],
    \end{aligned}
    \end{equation*}
    where $\|\mathscr{E}(t)\|\le C$ and $\|\mathscr{E}_\tau^n(t)\|\le C\tau$ is applied.
	Therefore, we derive that
	\begin{equation*}
	\begin{aligned}
	  \|G(t_n)-G^n\|&\le \big\|\big(\big((\mathscr{E}-\mathscr{E}_{\tau})*t\big)*f''\big)(t_n)\big\|\\
	  &\le\int_0^{t_n}\big\|\big((\mathscr{E}-\mathscr{E}_{\tau})*t\big)(t_n-s)\big\|\cdot\|f''(s)\|ds\\
  	  &\le C\tau^2\int_{0}^{t_n}\|f''(s)\|ds.
	\end{aligned}
	\end{equation*}
    This completes the proof.
	\end{proof}
	
	Combing the Taylor expansion $f(t)=f(0)+tf'(0)+(t*f'')(t)$ and the results in Theorems \ref{t4.2}, \ref{t4.3} and \ref{t4.5}, we can easily obtain the error estimate of inhomogeneous problem with nonsmooth initial data.
	\begin{theorem}\label{t4.6}
	Let $U(x)$ be bounded in $\bar{\Omega}$, $G(t_n)$ and $G^n$ be the solutions of \eqref{2.1} and \eqref{3.6} respectively, and $\tau\leq \tau_*$ with $\tau_*$ being a positive constant. If $G_0\in L^2(\Omega), f\in C^1([0,T],L^2(\Omega))$ and $\int_{0}^{t}\|f''(s)\|ds<\infty$, we have
	$$
	\|G(t_n)-G^n\|\le C\tau^2\Big(t_n^{-2}\|G_0\|+t_n^{-1}\|f(0)\|+\|f'(0)\|+\int_{0}^{t_n}\|f''(s)\|ds\Big).
	$$
	\end{theorem}

%-----------------------------------------------------------------------------------------------------
	\section{Numerical examples}\label{Section5}
%-----------------------------------------------------------------------------------------------------
	In this section, we perform some numerical experiments to verify the effectiveness of the proposed WSGD scheme \eqref{3.6} for solving the backward fractional Feynman-Kac equation \eqref{eq:bfFKe} with piecewise linear finite element approximation in space, and illustrate the theoretical analysis results in Section \ref{Section4}. Let $G_\tau$ be the numerical solution of $G$ at the final time $T$ with the step size $\tau$, we test the temporal error by
	\begin{equation*}
	E_\tau=\|G_\tau-G_{\tau/2} \|,
	\end{equation*}
	as the exact solution $G$ is unknown. Consequently, the temporal convergence rates can be calculated by the following formula
	\begin{equation*}
	\mathrm{Rate}=\frac{\ln(E_\tau/E_{\tau/2})}{\ln(2)}.
	\end{equation*}

	\subsection{One dimensional case}
	We first consider the one dimensional problem and $\Omega=(0,1)$. The spatial interval is equally divided into subintervals with a mesh size $h=1/128$ for the finite element discretization.
    \begin{example}\label{e5.1}
	We consider the homogeneous problem with nonsmooth initial data and verify the temporal convergence rate of the WSGD scheme \eqref{3.6}, and take $$
	\rho=-1+\mathbf{i},~ f(x,\rho,t)=0,~ G_0(x)=\chi_{(0,0.5)}(x),~ U(x)=\chi_{(0.5,1)}(x),~ T=1,
	$$
	where $\chi_{(a,b)}$ denotes the characteristic function over $(a,b)$,
	$$
	\chi_{(a,b)}(x)=\left\{\begin{aligned}
	&1,\ x\in (a,b),\\
	&0,\ \text{otherwise}.
	\end{aligned}\right.
	$$
    \end{example}
    
	\begin{table}[htpb]
	\caption{Temporal errors and convergence rates at $T$ for Example \ref{e5.1}.}\label{table1}
	\centering
	\begin{tabular}{cccccc}
	\toprule
	$\alpha\backslash\tau$&1/10 &1/20 &1/40 &1/80 &Rate	\\
	\midrule
	0.3& 4.8369E-05 & 1.1142E-05 & 2.6726E-06 & 6.5445E-07&$\approx$2.07(2.00)	\\
	0.5& 8.8909E-05 & 2.0464E-05 & 4.8980E-06 & 1.1978E-06&$\approx$2.07(2.00) \\
	0.7& 1.3446E-04 & 3.1062E-05 & 7.4008E-06 & 1.8043E-06&$\approx$2.07(2.00)	\\
	\bottomrule
	\end{tabular}
	\end{table}
	
	Taking $\alpha=0.3, 0.5, 0.7$ and $\tau=1/10, 1/20, 1/40, 1/80$, Table \ref{table1} presents the temporal $L^2$ errors and convergence rates at $T$ for Example \ref{e5.1} by the WSGD scheme \eqref{3.6}. We observe that the numerical temporal convergence rate is of second order for all three fractional orders, which is consistent with the theoretical result in Theorem \ref{t4.2}.

    \begin{example}\label{e5.2}
	We consider the inhomogeneous problem with zero initial data, and take $$
    \rho=-1,~ f(x,\rho,t)=x(1-x)e^{-t\rho\chi_{(0.5,1)}(x) },~ G_0(x)=0,~ U(x)=\chi_{(0.5,1)}(x),~ T=1.
    $$
    \end{example}

	\begin{table}[htpb]
	\caption{Temporal errors and convergence rates at $T$ for Example \ref{e5.2}.}
	\label{table3}
	\centering
	\begin{tabular}{cccccc}
	\toprule
	$\alpha\backslash\tau$&1/10 &1/20 &1/40 &1/80 &Rate	\\
	\midrule
	0.3& 4.7712E-05 & 1.2368E-05 & 3.1461E-06 & 7.9328E-07&$\approx$1.97(2.00)	\\
	0.5& 2.4018E-05 & 6.3733E-06 & 1.6370E-06 & 4.1458E-07&$\approx$1.95(2.00) \\
	0.7& 4.7152E-06 & 1.2328E-06 & 3.2488E-07 & 8.3728E-08&$\approx$1.94(2.00) \\
	\bottomrule
	\end{tabular}
	\end{table}

	The temporal convergence of the WSGD scheme \eqref{3.6} is verified for Example \ref{e5.2} with $\alpha=0.3, 0.5, 0.7$ and $\tau=1/10, 1/20, 1/40, 1/80$. The temporal errors and convergence rates at $T$ are presented in Table \ref{table3}. Second-order convergence rate is also observed, which verifies the result of Theorem \ref{t4.6}.

    \begin{example}\label{e5.3}
    The inhomogeneous problem with nonsmooth initial data is considered by taking $$
    \rho=-1,~ f(x,\rho,t)=x(1-x)e^{-t\rho\chi_{(0.5,1)}(x) },~ G_0(x)=\chi_{(0,0.5)}(x),~ U(x)=\chi_{(0.5,1)}(x),~ T=1.
    $$
    \end{example}

	\begin{table}[htpb]
	\caption{Temporal errors and convergence rates at $T$ for Example \ref{e5.3}.}
	\label{table5}
	\centering
	\begin{tabular}{cccccc}
	\toprule
	$\alpha\backslash\tau$&1/10 &1/20 &1/40 &1/80 &Rate	\\
	\midrule
	0.3& 9.3834E-05 & 2.3020E-05 & 5.7054E-06 & 1.4205E-06& $\approx$2.01(2.00)	\\
	0.5& 1.1269E-04 & 2.6875E-05 & 6.5344E-06 & 1.6128E-06& $\approx$2.04(2.00) \\
	0.7& 1.4131E-04 & 3.3013E-05 & 7.9159E-06 & 1.9364E-06& $\approx$2.06(2.00) \\
	\bottomrule
	\end{tabular}
	\end{table}
	
    For $\alpha=0.3, 0.5, 0.7$ and $\tau=1/10, 1/20, 1/40, 1/80$, the temporal errors and convergence rates at $T$ by the WSGD scheme \eqref{3.6} for Example \ref{e5.3} are presented in Table \ref{table5}, which shows the second-order convergence rate and agrees well with the theoretical result in Theorem \ref{t4.6}. The above numerical results illustrate that our proposed scheme \eqref{3.6} is effective for numerically solving the backward fractional Feynman-Kac equation \eqref{eq:bfFKe} in one dimensional case.

    \subsection{Two dimensional case}
    Next we show the numerical efficiency of the WSGD scheme \eqref{3.6} for solving the backward fractional Feynman-Kac equation \eqref{eq:bfFKe} in two dimensional case and verify the theoretical convergence results. The spatial domain $\Omega=(0,1)^2$ is considered and uniformly partitioned into triangles with the mesh size $h=1/128$ for the finite element discretization.
    \begin{example}\label{e5.4}
    We set $T=1$, $\rho=-1$, $x=(x_1,x_2)$,
    $$
    f(x,\rho,t)=x_1(1-x_1)x_2(1-x_2)e^{-t\rho U(x) },~ G_0(x)=\chi_{(0,0.5)\times(0,0.5)}(x_1,x_2),
    $$
    and consider three cases of $U(x)$ as follows
    $$
    \left\{
    \begin{aligned}
    &\text{case}\ \text{(a)}:U(x)=\chi_{(0.5,1)\times(0.5,1)}(x_1,x_2),\\
    &\text{case}\ \text{(b)}:U(x)=x_1+x_2,\\
    &\text{case}\ \text{(c)}:U(x)=x_1^2+x_2^2,\\
    \end{aligned}\right.
    $$
    where $\chi_{(a,b)\times(c,d)}(x_1,x_2)$ denotes the characteristic function on $(a,b)\times(c,d)$, i.e.
    $$
    \chi_{(a,b)\times(c,d)}(x_1,x_2)=\left\{\begin{aligned}
    &1,\ (x_1,x_2)\in (a,b)\times(c,d),\\
    &0,\ \text{otherwise}.
    \end{aligned}\right.
    $$
    \end{example}
    
    \begin{table}[H]
    \caption{Temporal errors and convergence rates at $T$ for Example \ref{e5.4}.}
    \label{table7}
    \centering
    \begin{tabular}{ccccccc}
    \toprule
    case & $\alpha\backslash\tau$ & 1/10 & 1/20 & 1/40 & 1/80 & Rate\\
    \midrule
    \multirow{2}{*}{(a)}&0.2&1.0561E-05&2.5062E-06&6.1101E-07&1.5089E-07&$\approx$2.04(2.00) \\
    &0.8&2.6980E-05 &6.2089E-06&1.4771E-06 & 3.5962E-07& $\approx$2.07(2.00)\\
    \midrule
    \multirow{2}{*}{(b)}&0.2&2.7957E-05&6.7950E-06&1.6771E-06&4.1676E-07&$\approx$2.02(2.00) \\
    &0.8&5.3769E-05 &1.2412E-05&2.9587E-06 & 7.2115E-07& $\approx$2.07(2.00)\\
    \midrule
    \multirow{2}{*}{(c)}&0.2&1.4881E-05&3.5516E-06&8.6835E-07&2.1475E-07&$\approx$2.03(2.00) \\
    &0.8&3.4592E-05 &7.9669E-06&1.8963E-06 & 4.6183E-07& $\approx$2.07(2.00)\\
    \bottomrule
    \end{tabular}
    \end{table}
    
    The temporal errors and convergence rates at $T$ by the WSGD scheme \eqref{3.6} for Example \ref{e5.4} are presented in Table \ref{table7} with $\alpha=0.2, 0.8$ and $\tau=1/10, 1/20, 1/40, 1/80$, which is also consistent with the result of Theorem \ref{t4.6}. It reveals that the proposed WSGD scheme performs effectively and converges numerically by the theoretical rate for the backward fractional Feynman-Kac equation \eqref{eq:bfFKe} in two dimensional case.

%-----------------------------------------------------------------------------------------------------
	\section{Conclusion}\label{Section6}
%-----------------------------------------------------------------------------------------------------
    The fractional Feynman-Kac equation \eqref{eq:bfFKe} has a wide of applications for describing the distribution of the particles' path functionals in the anomalous diffusion process. The time-space coupled non-local operator and complex parameters in the equation bring challenges on numerical analysis and computations. We propose an efficient second-order time-stepping scheme \eqref{3.6} for solving \eqref{eq:bfFKe} based on the weighted and shifted Gr\"{u}nwald difference formula, which also weakens the regularity requirement of the function $U(x)$. The error estimates are rigorously analyzed for the inhomogeneous problem with nonsmooth initial data, depending only on the regularity assumptions on the data, without the requirement of regularity of the exact solution. The temporal convergence rate $O(\tau^2)$ is verified by some numerical examples.
%-----------------------------------------------------------------------------------------------------
%\vskip12pt\noindent
%{\bf Data Availability} Enquiries about data availability should be directed to the authors.
%
%\vskip12pt\noindent
%{\bf\Large Declarations}
%
%\vskip8pt\noindent
%{\bf Competing interests} The authors have not disclosed any competing interests.
%-----------------------------------------------------------------------------------------------------
%\bibliographystyle{abbrv}%jsc
%\bibliography{../../tianwy}

\begin{thebibliography}{10}

\bibitem{CarmiTB:2010}
S.~Carmi, L.~Turgeman, and E.~Barkai.
\newblock On distributions of functionals of anomalous diffusion paths.
\newblock {\em J. Stat. Phys.}, 141(6):1071--1092, 2010.

\bibitem{ChenD:2015b}
M.~Chen and W.~Deng.
\newblock High order algorithms for the fractional substantial diffusion
  equation with truncated {L}\'evy flights.
\newblock {\em SIAM J. Sci. Comput.}, 37(2):A890--A917, 2015.

\bibitem{ChenD:2015a}
M.~Chen and W.~H. Deng.
\newblock Discretized fractional substantial calculus.
\newblock {\em ESAIM Math. Model. Numer. Anal.}, 49(2):373--394, 2015.

\bibitem{ChenSZZ:2020}
S.~Chen, J.~Shen, Z.~Zhang, and Z.~Zhou.
\newblock A spectrally accurate approximation to subdiffusion equations using
  the log orthogonal functions.
\newblock {\em SIAM J. Sci. Comput.}, 42(2):A849--A877, 2020.

\bibitem{DengHWX:2020}
W.~Deng, R.~Hou, W.~Wang, and P.~Xu.
\newblock {\em Modeling Anomalous Diffusion: From Statistics to Mathematics}.
\newblock World Scientific, 2020.

\bibitem{DengCB:2015}
W.~H. Deng, M.~H. Chen, and E.~Barkai.
\newblock Numerical algorithms for the forward and backward fractional
  {F}eynman-{K}ac equations.
\newblock {\em J. Sci. Comput.}, 62(3):718--746, 2015.

\bibitem{DengLQW:2018}
W.~H. Deng, B.~Li, Z.~Qian, and H.~Wang.
\newblock Time discretization of a tempered fractional {F}eynman-{K}ac equation
  with measure data.
\newblock {\em SIAM J. Numer. Anal.}, 56(6):3249--3275, 2018.

\bibitem{FanLS:2023}
E.~Fan, C.~Li, and M.~Stynes.
\newblock Discretised general fractional derivative.
\newblock {\em Math. Comput. Simulation}, 208:501--534, 2023.

\bibitem{GunzburgerLW:2019b}
M.~Gunzburger, B.~Li, and J.~Wang.
\newblock Convergence of finite element solutions of stochastic partial
  integro-differential equations driven by white noise.
\newblock {\em Numer. Math.}, 141(4):1043--1077, 2019.

\bibitem{GunzburgerW:2019b}
M.~Gunzburger and J.~Wang.
\newblock A second-order {C}rank-{N}icolson method for time-fractional {PDE}s.
\newblock {\em Int. J. Numer. Anal. Model.}, 16(2):225--239, 2019.

\bibitem{HaoCL:2017}
Z.~Hao, W.~Cao, and G.~Lin.
\newblock A second-order difference scheme for the time fractional substantial
  diffusion equation.
\newblock {\em J. Comput. Appl. Math.}, 313:54--69, 2017.

\bibitem{JinLZ:2019}
B.~Jin, R.~Lazarov, and Z.~Zhou.
\newblock Numerical methods for time-fractional evolution equations with
  nonsmooth data: A concise overview.
\newblock {\em Comput. Methods Appl. Mech. Engrg.}, 346:332--358, 2019.

\bibitem{JinLZ:2017}
B.~Jin, B.~Li, and Z.~Zhou.
\newblock Correction of high-order {BDF} convolution quadrature for fractional
  evolution equations.
\newblock {\em SIAM J. Sci. Comput.}, 39(6):A3129--A3152, 2017.

\bibitem{JinLZ:2018a}
B.~Jin, B.~Li, and Z.~Zhou.
\newblock An analysis of the {C}rank-{N}icolson method for subdiffusion.
\newblock {\em IMA J. Numer. Anal.}, 38(1):518--541, 2018.

\bibitem{LiLX:2019}
B.~Li, H.~Luo, and X.~Xie.
\newblock Analysis of a time-stepping scheme for time fractional diffusion
  problems with nonsmooth data.
\newblock {\em SIAM J. Numer. Anal.}, 57(2):779--798, 2019.

\bibitem{LiDZ:2019}
C.~Li, W.~Deng, and L.~Zhao.
\newblock Well-posedness and numerical algorithm for the tempered fractional
  differential equations.
\newblock {\em Discrete Contin. Dyn. Syst. Ser. B}, 24(4):1989--2015, 2019.

\bibitem{LiZL:2012}
C.~Li, F.~Zeng, and F.~Liu.
\newblock Spectral approximations to the fractional integral and derivative.
\newblock {\em Fract. Calc. Appl. Anal.}, 15(3):383--406, 2012.

\bibitem{LiX:2009}
X.~J. Li and C.~J. Xu.
\newblock A space-time spectral method for the time fractional diffusion
  equation.
\newblock {\em SIAM J. Numer. Anal.}, 47(3):2108--2131, 2009.

\bibitem{LinX:2007}
Y.~Lin and C.~Xu.
\newblock Finite difference/spectral approximations for the time-fractional
  diffusion equation.
\newblock {\em J. Comput. Phys.}, 225(2):1533--1552, 2007.

\bibitem{Lubich:1986b}
C.~Lubich.
\newblock Discretized fractional calculus.
\newblock {\em SIAM J. Math. Anal.}, 17(3):704--719, 1986.

\bibitem{LubichST:1996}
C.~Lubich, I.~H. Sloan, and V.~Thom\'{e}e.
\newblock Nonsmooth data error estimates for approximations of an evolution
  equation with a positive-type memory term.
\newblock {\em Math. Comp.}, 65(213):1--17, 1996.

\bibitem{MustaphaAF:2014}
K.~Mustapha, B.~Abdallah, and K.~M. Furati.
\newblock A discontinuous {P}etrov-{G}alerkin method for time-fractional
  diffusion equations.
\newblock {\em SIAM J. Numer. Anal.}, 52(5):2512--2529, 2014.

\bibitem{Podlubny:1999}
I.~Podlubny.
\newblock {\em Fractional Differential Equations}.
\newblock Academic Press, San Diego, 1999.

\bibitem{StynesOG:2017}
M.~Stynes, E.~O'Riordan, and J.~L. Gracia.
\newblock Error analysis of a finite difference method on graded meshes for a
  time-fractional diffusion equation.
\newblock {\em SIAM J. Numer. Anal.}, 55(2):1057--1079, 2017.

\bibitem{SunND:2020}
J.~Sun, D.~Nie, and W.~Deng.
\newblock Error estimates for backward fractional {F}eynman-{K}ac equation with
  non-smooth initial data.
\newblock {\em J. Sci. Comput.}, 84(1):Paper No. 6, 23, 2020.

\bibitem{SunND:2021a}
J.~Sun, D.~Nie, and W.~Deng.
\newblock High-order {BDF} fully discrete scheme for backward fractional
  {F}eynman-{K}ac equation with nonsmooth data.
\newblock {\em Appl. Numer. Math.}, 161:82--100, 2021.

\bibitem{TianZD:2015}
W.~Y. Tian, H.~Zhou, and W.~H. Deng.
\newblock A class of second order difference approximation for solving space
  fractional diffusion equations.
\newblock {\em Math. Comp.}, 84(294):1703--1727, 2015.

\bibitem{TurgemanCB:2009}
L.~Turgeman, S.~Carmi, and E.~Barkai.
\newblock Fractional {F}eynman-{K}ac equation for non-{B}rownian functionals.
\newblock {\em Phys. Rev. Lett.}, 103(19):190201, 2009.

\bibitem{WangYYP:2020}
Y.~Wang, Y.~Yan, Y.~Yan, and A.~K. Pani.
\newblock Higher order time stepping methods for subdiffusion problems based on
  weighted and shifted {G}r\"{u}nwald-{L}etnikov formulae with nonsmooth data.
\newblock {\em J. Sci. Comput.}, 83(3):Paper No. 40, 29, 2020.

\bibitem{YanKF:2018}
Y.~Yan, M.~Khan, and N.~J. Ford.
\newblock An analysis of the modified {L}1 scheme for time-fractional partial
  differential equations with nonsmooth data.
\newblock {\em SIAM J. Numer. Anal.}, 56(1):210--227, 2018.

\bibitem{ZhangSL:2014}
Y.-N. Zhang, Z.-Z. Sun, and H.-L. Liao.
\newblock Finite difference methods for the time fractional diffusion equation
  on non-uniform meshes.
\newblock {\em J. Comput. Phys.}, 265:195--210, 2014.

\bibitem{ZhengLTA:2015}
M.~Zheng, F.~Liu, I.~Turner, and V.~Anh.
\newblock A novel high order space-time spectral method for the time fractional
  {F}okker-{P}lanck equation.
\newblock {\em SIAM J. Sci. Comput.}, 37(2):A701--A724, 2015.

\bibitem{ZhouT:2022}
H.~Zhou and W.~Y. Tian.
\newblock Two time-stepping schemes for sub-diffusion equations with singular
  source terms.
\newblock {\em J. Sci. Comput.}, 92(2):Paper No. 70, 28, 2022.

\end{thebibliography}

\end{document}